\newif\ifdraft
\draftfalse

\documentclass[10pt, reqno]{amsart}
\evensidemargin
\oddsidemargin
\makeatletter\@addtoreset{equation}{section}\makeatother

\usepackage{amsmath, amsthm, latexsym, amssymb, bbm, graphicx, subcaption}
\usepackage{listings} 
\usepackage{comment} 
\usepackage{xcolor}
\usepackage[citecolor=blue]{hyperref}
\hypersetup{
	colorlinks=true,
	linkcolor=blue,
	urlcolor=blue,
}
\newcounter{sideremark}
\newcommand{\marrow}[1]{\stepcounter{sideremark}\marginpar{$\color{#1}\boldsymbol{\longleftarrow\scriptstyle    \color{#1}\arabic{sideremark}}$}}

\newcommand{\nnote}[1]{
	\ifdraft{
		\textsf{{\color{red}*** (Nadia) \marrow{red} #1 ***}}}\fi}

\newcommand{\gnote}[1]{
	\ifdraft{
		\textsf{{\color{blue}*** (Giordano) \marrow{blue} #1 ***}}}\fi}

\definecolor{dkgreen}{rgb}{0,0.6,0}
\definecolor{gray}{rgb}{0.5,0.5,0.5}
\definecolor{mauve}{rgb}{0.58,0,0.82}

\newcommand{\e}{\varepsilon}

\renewcommand{\phi}{\varphi}

\newcommand{\ssup}[1] {{\scriptscriptstyle{({#1}})}}
\newcommand{\sssup}[1] {{\scriptscriptstyle{[{#1}}]}}
\newcommand{\vertiii}[1]{{\left\vert\kern-0.25ex\left\vert\kern-0.25ex\left\vert#1\right\vert\kern-0.25ex\right\vert\kern-0.25ex\right\vert}}
\newcommand{\vertiiiBig}[1]{{\Big\vert\kern-0.25ex\Big\vert\kern-0.25ex\Big\vert#1\Big\vert\kern-0.25ex\Big\vert\kern-0.25ex\Big\vert}}
\newcommand{\vertiiibig}[1]{{\big\vert\kern-0.25ex\big\vert\kern-0.25ex\big\vert#1\big\vert\kern-0.25ex\big\vert\kern-0.25ex\big\vert}}

\newcommand{\one}{{\mathbf 1}}
\newcommand{\1}{\mathbbm{1}}

\newcommand{\R}{\mathbb R}

\newcommand{\N}{\mathbb N}

\newcommand{\E}{\mathbb E}
\renewcommand{\P}{\mathbb P}

\newtheorem{theorem}{Theorem}
\newtheorem{lemma}{Lemma}[section]

\newtheorem{prop}[lemma]{Proposition}

{\nopagebreak {\hfill{$\diamond$}}\\ }
{\nopagebreak {\hfill{$\diamond$}}\\ }

\renewcommand{\phi}{\varphi}

\renewcommand{\P}{\mathbb{P}}
\renewcommand{\E}{\mathbb{E}}

\begin{document}

\title{Edge-reinforced branching random walk on a triangle}

\author[Giordano Giambartolomei and Nadia Sidorova]{}

\maketitle

\centerline{\sc Giordano Giambartolomei\footnote{Department of Informatics, King's College London, Bush House, London WC2B 4BG, UK, {\tt giordano.giambartolomei@kcl.ac.uk}.} and Nadia Sidorova\footnote{Department of Mathematics, University College London, Gower Street, London WC1 E6BT, UK, {\tt n.sidorova@ucl.ac.uk}.
} }

\vspace{0.4cm}


%
\vspace{0.4cm}

\begin{quote}
{\small {\bf Abstract:} 
The edge-reinforced random walk (ERRW) is a random process on the vertices of a graph that is more likely to cross the edges it has visited in the past. Depending on the strength of the reinforcement,  the ERRW of a single particle can either exhibit localisation (eventually moving back and forth across a single edge) or remain transient. We consider a model where a single ERRW is replaced by that of an exponentially growing number of random particles, and we study its localisation properties on the triangle. Using the dynamical systems approach we analyse the frequencies with which the edges are traversed and prove their almost sure convergence. We discuss the scenarios when those frequencies become negligible for one or two edges (dominance). We also discuss the situation when an edge stops being traversed entirely (monopoly).

}
\end{quote}
\vspace{5ex}

{\small {\bf AMS Subject Classification:} 
60K35, 60K37, 37A50, 60J80.

}
\vspace{4ex}



\bigskip

\section{Introduction}

\subsection{The model and results.}

The edge-reinforced random walk (ERRW) is a variant of a random walk on a graph, introduced by
Coppersmith and Diaconis in~\cite{CopDiac87}
, where the probability of traversing an edge increases each time that edge is used.
One of the most natural reinforcement mechanisms is the linear reinforcement where that probability depends linearly on the number 
of traversals. In this paper we consider the linear ERRW of not one but an exponentially growing family of particles, driven by a Galton-Watson process, on a triangle graph.
\smallskip

Let $\N_0=\N\cup \{0\}$ and let $\R_+$ be the set of strictly positive reals. We use the numbers $I=\{1,2,3\}$ to mark the vertices and the edges of the triangle graph, and we use the convention that the $i$-th edge is the one not adjacent to the $i$-th vertex. 
We denote by $\oplus$ and $\ominus$ the addition and subtraction modulo three on $I$, respectively. We use the probability space $(\Omega,\mathcal{F}, \P)$ and denote by $\E$ the corresponding expectation. 
\smallskip

For each $i\in I$ and $n\in \N_0$, we denote by $N_n^{\ssup i}$ the number of particles at the vertex $i$ at the time $n$, and 
by $T_n^{\ssup i}$ the number of traversals of the edge $i$ by all particles by the time $n$, plus some initial weights. 
\smallskip

Let $Z$ be the offspring random variable with the mean $\mu=\E(Z)\in (1,\infty)$ and with the offspring distribution satisfying $\P(Z=0)=0$
and $\E(Z^2)<\infty$.
\smallskip

At time $n$, each particle is replaced by a random number of particles driven by the offspring distribution $Z$, independently of the other particles and of the past. After that, each particle moves to one of the two adjacent vertices choosing between them with the probabilities proportional to the number of previous traversals by all particles plus the initial weights. We define this inductively as follows.
\smallskip

At time $n=0$, let $\mathcal{F}_0=\{\emptyset,\Omega\}$. 
Let $N_0\in \N_0^3$ be the deterministic initial number of particles at the vertices and $T_0\in \R_+^3$ be the initial weights of the edges. 
We assume that we start with at least one particle, i.e.\ $N_0\neq (0,0,0)$. Let $(Z^{\ssup i}_{1,k})$, $i\in I$, $k\in\N$, be a collection of 
independent random variables with the offspring distribution $Z$. Those will be responsible for the first branching. 
Denoting $Z_{1,k}=(Z^{\ssup i}_{1,k})_{i\in I}$, 
let  $\mathcal{F}_0^{\star}=\sigma(\mathcal{F}_0, (Z_{1,k})_{k\in\N})$ be the $\sigma$-algebra after the first branching but 
before the first movement. 
\smallskip

Now consider the transition between the times $n$ and $n+1$. 
For each $i\in I$, conditionally on the past $\mathcal{F}_n^{\star}$, let
\begin{align*}
B_{n+1}^{\ssup i}\sim \text{Bin}\Big(\sum_{k=1}^{N_n^{\ssup i}}Z^{\ssup i}_{n+1,k},P_n^{\ssup i}\Big)
\end{align*}
be Binomial random variables with 
\begin{align}
\label{p}
P_n^{\ssup i}= \frac{T^{\ssup{i \ominus 1}}_n}{T^{\ssup{i \ominus 1}}_n+T^{\ssup{i \oplus 1}}_n}
\end{align}
otherwise independent from the past and from each other. They determine how many particles will move from the vertex $i$
to the vertex $i\ominus 1$. To benefit from symmetry, we will also use 
\begin{align*}
\bar B_{n+1}^{\ssup i}=\sum_{k=1}^{N_n^{\ssup i}}Z^{\ssup i}_{n+1,k}-B_{n+1}^{\ssup i},
\end{align*}
which gives the number of particles moving from the vertex $i$
to the vertex $i\oplus 1$. We obviously have 
\begin{align*}
\bar B_{n+1}^{\ssup i}\sim\text{Bin}\Big(\sum_{k=1}^{N_n^{\ssup i}}Z^{\ssup i}_{n+1,k},\bar P_n^{\ssup i}\Big)
\qquad\text{with}\qquad\bar P_n^{\ssup i}=\frac{T^{\ssup{i \oplus 1}}_n}{T^{\ssup{i \ominus 1}}_n+T^{\ssup{i \oplus 1}}_n}
\end{align*}
and $P_n^{\ssup i}+\bar P_n^{\ssup i}=1$. The updated number of traversals and number of particles are then given by  
\begin{align}
\label{def_t}
T^{\ssup i}_{n+1}&=T_n^{\ssup i}+B_{n+1}^{\ssup{i\oplus 1}}+\bar B_{n+1}^{\ssup{i\ominus 1}},\\
N^{\ssup i}_{n+1}&=B_{n+1}^{\ssup{i\ominus 1}}+\bar B_{n+1}^{\ssup{i\oplus 1}}
\label{def_n}
\end{align}
and, denoting $B_{n+1}=(B_{n+1}^\ssup{i})_{i\in I}$, the $\sigma$-algebra after the $(n+1)$-st movement, but before the $(n+2)$-nd branching is $\mathcal{F}_{n+1}=\sigma(\mathcal{F}_n^{\star},B_{n+1})$.
Further, let $(Z^{\ssup i}_{n+2,k})$, $i\in I$, $k\in\N$,  be a collection of random variables which are all independent of each other and of $\mathcal{F}_{n+1}$ and have the offspring distribution $Z$. Finally, denoting $Z_{n+2,k}=(Z^{\ssup i}_{n+2,k})_{i\in I}$, let $\mathcal{F}_{n+1}^{\star}=\sigma(\mathcal{F}_{n+1}, (Z_{n+2,k})_{k\in \N})$ be the $\sigma$-algebra after the 
$(n+2)$-nd branching but before the $(n+2)$-nd movement.
\smallskip

We are interested in the frequency with which each edge is traversed and the proportion of particle at each vertex so we define 
\begin{align}
\label{def_thpi}
\Theta_n^{\ssup i}=\frac{T_n^{\ssup i}}{\Vert  T_n\Vert }
\qquad\text{and}\qquad
\pi_n^{\ssup i}=\frac{N_n^{\ssup i}}{\Vert  N_n\Vert }, \qquad  i\in I, n\in\N,
\end{align}
where $\Vert \cdot\Vert $ denotes the $\ell_1$-norm on $\R^3$ (and we will 
denote by the same symbol the corresponding matrix norm on the space of tree-dimensional square matrices). Note that for all $n$, $\Vert  T_{n+1}\Vert =\Vert  T_n\Vert +\Vert  N_{n+1}\Vert $. Obviously, $\Theta_n=(\Theta_n^{\ssup i})_{i\in I}$ and 
$\pi_n=(\pi_n^{\ssup i})_{i\in I}$ belong to 
the unit simplex 
$$\Sigma=\{p\in [0,1]^3: \Vert  p\Vert =1\}.$$  
Denote by  
$V=\{v_1,v_2,v_3\}$ the set of three vertices of $\Sigma$ and by $\partial\Sigma$ its triangular boundary. 
\smallskip

For each $p=(x,y,z)\in \Sigma\setminus V$, denote
\begin{align*}
M_p=\left(\begin{array}{ccc}
0 & \frac{z}{x+z} & \frac{y}{x+y} \\
\frac{z}{y+z} & 0 & \frac{x}{x+y} \\
\frac{y}{y+z} & \frac{x}{x+z} & 0
\end{array}\right),
\end{align*}
where the shape of the matrix is inherited from the transition probabilities~\eqref{p}. 
Further, let 
\begin{align}
\label{qp}
q_{[p]}=\frac{1}{2}(\one-p),
\end{align}
where $\one=(1,1,1)$. 
Observe that $q_{[p]}$ is an eigenvector of $M_p$ with eigenvalue $1$. Further, if 
$p\in \partial\Sigma\setminus V$ then $M_p$ also has eigenvalues $0$ and $-1$ since $\text{det} (M_p)=0$
and $\text{tr}(M_p)=0$. It is easy to see that the normalised eigenvector corresponding to the eigenvalue $-1$
is given by 
\begin{align}
e_{-1}(p)=\left\{\begin{array}{ll}
\frac 1 2 (y, x, -1) & \text{if } p=(x,y,0),\\
\frac 1 2 (z, -1, x) & \text{if } p=(x,0,z),\\
\frac 1 2 (-1, z, y) & \text{if } p=(0,y,z).\\
\end{array}\right.
\label{e1}
\end{align}
If $p\in V$ then $M_p$ is not well-defined, but we will still define 
\begin{align*}
e_{-1}(v_i)=\frac 1 2 (v_{i\oplus 1}-v_{i\ominus 1}), \quad i\in I.
\end{align*}
Observe that while the 
eigenspace of the eigenvalue $-1$ continuously depends on $p$ on $\Sigma$, the eigenvector is discontinuous at the vertices due to the change of direction. 
\smallskip

For any $p\in\Sigma$, we call $(p,q_{[p]})$ an equilibrium point. Further, for any $p\in\Sigma$ and $\theta\in \R$, we call 
the pair $\{(p,q_{[p]}\pm\theta e_{-1}(p))\}$ 
an equilibrium two-cycle of amplitude $\theta$.   
We say that a sequence $(u_n)$ approaches a two-cycle $\{a,b\}$ if 
either $u_{2n}\to a$, $u_{2n+1}\to b$ or $u_{2n}\to b$, $u_{2n+1}\to a$. 
\smallskip

We are interested in the dynamics of $(\Theta_n,\pi_n)$, and the first step towards this is the following lemma, which will be proved in Section~\ref{s_prelim}.

\begin{lemma} 
\label{lemma_l}
As $n\to\infty$, $\Vert \pi_n-q_{[\Theta_n]}\Vert $ converges almost surely to a 
random variable $\ell$. 
\end{lemma}

The next theorem
states that $(\Theta_n,\pi_n)$ either converges to an equilibrium or approaches an equilibrium two-cycle of amplitude $\ell$. 
 
\begin{theorem} 
\label{th_conv}
As $n\to\infty$,  $\Theta_n$ converges almost surely to a $\Sigma$-valued random variable $\Theta$, 
and one of the following occurs: 

\begin{enumerate}

\item $\Theta\in \Sigma\setminus \partial\Sigma$, $\ell=0$ and $(\Theta_n, \pi_n)\to (\Theta, q_{[\Theta]})$
\hfill internal equilibrium

\item $\Theta\in \partial\Sigma\setminus V$, $\ell=0$ and $(\Theta_n, \pi_n)\to (\Theta, q_{[\Theta}])$
\hfill edge equilibrium

\item $\Theta\in \partial\Sigma\setminus V$, $\ell>0$ and $(\Theta_n, \pi_n)$ approaches $\{(\Theta, q_{[\Theta]}\pm \ell e_{-1}(\Theta))\}$
\hfill edge equilibrium two-cycle

\item $\Theta\in V$, $\, \ \ \ \ \ \ell=0$ and $(\Theta_n, \pi_n)\to (\Theta, q_{[\Theta]})$
\hfill vertex equilibrium

\item $\Theta\in V$, $\, \ \ \ \ \ \ell>0$ and $(\Theta_n, \pi_n)$ approaches $\{(\Theta, q_{[\Theta]}\pm \ell e_{-1}(\Theta))\}$
\hfill vertex equilibrium two-cycle

\end{enumerate}

\end{theorem}

We see that if $\Theta$ is an internal point then all edges get a non-negligible proportion of traversals, and the population masses $\pi_n$
are forced to converge. If $\Theta$ lies on the boundary (excluding the vertices), 
then the traversals of exactly one edge become negligible. 
The masses $\pi_n$ will converge 
if the population is balanced roughly equally between the vertices adjacent to the 
unpopular edge and the remaining vertex; otherwise $\pi_n$ will fluctuate between two states with the amplitude reflecting the size of the imbalance. Finally, if $\Theta$ is a vertex then the overwhelming number of traversals occurs over a single edge. 
In that case $q_{[\Theta]}\in \Sigma\setminus V$ belongs to an edge and $e_{-1}(\Theta)$ is parallel to that edge, meaning that  
the mass of the vertex between the unpopular edges becomes negligible; the remaining two masses either converge to a half each, or fluctuate 
between two states in the presence of imbalance. 
\smallskip

We are naturally curious to find out which of the five scenarios  actually happen. Borrowing from the terminology established for balls and bins models \cite{DriFriMitz02, Sid18}, we call the events 
\begin{align*}
\mathcal{D}_{0}=\{\Theta\in \Sigma\setminus \partial\Sigma\}
\qquad\qquad 
\mathcal{D}_{p}=\{\Theta\in \partial\Sigma\setminus V\},
\qquad\qquad 
\mathcal{D}=\{\Theta\in V\},
\end{align*}
\emph{no dominance}, \emph{partial dominance} and \emph{full dominance}, respectively.

\begin{theorem} 
\label{th_pos}
$\P(\mathcal{D}_{0})>0$ and $\P(\mathcal{D}_{p})>0$. 
\end{theorem}

We actually prove a slightly stronger result, namely, that scenarios (1) and (3) in Theorem~\ref{th_conv} occur with positive probability. 
We also conjecture that the remaining scenarios (2), (4) and (5) almost surely do not occur, and that in particular $\P(\mathcal{D})=0$.  
\smallskip



Further insight into the behaviour of $(\Theta_n, \pi_n)$ will be 
provided in Figure~\ref{fig:twoimages} in Section~\ref{further}.
\smallskip

Finally, we are interested in examining the extent of how unpopular the edges can be, and whether the traversals can cease to happen 
completely.  With that in mind, we denote by $J$ be the number of indices $i$ such that $T_n^{\ssup i}$ is bounded and, again following the balls and bins jargon, call the events 
\begin{align*}
\mathcal{M}_{0}=\{J=0\}
\qquad\qquad 
\mathcal{M}_{p}=\{J=1\},
\qquad\qquad 
\mathcal{M}=\{J=2\},
\end{align*}
\emph{no monopoly}, \emph{partial monopoly} and \emph{full monopoly}, respectively. The following theorem states that 
our model is not monopolistic in any way. 

\begin{theorem} 
\label{th_mon}
$\P(\mathcal{M}_0)=1$ and $\P(\mathcal{M}_p)=\P(\mathcal{M})=0$. Moreover, almost surely $T_n^{\ssup i}$ grows 
super-polynomially for each $i$. 
\end{theorem}
\medskip

\subsection{Related work and motivation}

Reinforced random processes have seen significant progress over the past thirty years (see~\cite{Pem07} for a survey). Here, we review recent advances on ERRWs that motivate our work (see also~\cite{MerRoll06, Koz14} for detailed surveys). 
\smallskip

In an ERRW a particle moves on the vertices of a weighted graph, selecting a neighbour at random with probability proportional to the weight of the incident edge. 
Coppersmith and Diaconis were the first to study 
a \emph{linear} ERRW (abbreviated LERRW), where the weight of the edge increases by one with each traversal~\cite{CopDiac87,Diac88'}. They showed that 
on a \emph{finite graph}
the process exhibits almost sure recurrence and convergence of the normalised edge occupation vector to a random limit with an explicitly known density.
\smallskip

A substantial amount of study has been spent on almost sure recurrence and transience criteria for LERRWs on \emph{infinite graphs}. An important feature of the LERRW is that it is a \emph{partially exchangeable} process and hence 
a mixture of Markov chains~\cite{Roll03,MerRoll07}.
After early results~\cite{Pem88',Roll06}, 
Sabot and Tarr\`{e}s proved that recurrence holds in any dimension $d$ for $\mathbb{Z}^d$, provided the initial weights are small enough~\cite{SabTar15}. Notably, they relied on a connection with certain quantum models that allowed to exploit partial exchangeability\footnote{Recurrence can also be proven without relying on this connection~\cite{AngCrawKoz14}.}.
\smallskip

Stronger than linear reinforcement for random walks (SERRW) has been studied in the models, where transition probabilities are proportional to a certain positive function $f$ of the number of edge crossings. In~\cite{Dav90} Davis showed that if the function $f$ grows fast enough, such that
\begin{equation}
\sum_{n=1}^{\infty} \frac{1}{f(n)}<\infty,\tag*{(H)}\label{H}
\end{equation} 
then the walk on $\mathbb{Z}$ almost surely gets eventually trapped in a \emph{single edge}; if the series diverges the walk is recurrent almost surely. 
Soon afterwards in~\cite{Sel94}
Sellke proved the single edge localisation on $\mathbb{Z}^d$. 
Although his argument extended to any bipartite graph of bounded degree, it failed on graphs with odd cycles, including the triangle graph as 
a crucial case. He conjectured, however, that localisation should still hold, and this became known as \emph{Sellke's conjecture}. Significant progress towards its resolution was made by Limic, who proved it for 
$f(n)=n^{\alpha}$, with the \emph{feedback} parameter $\alpha>1$, on any graph, see~\cite{Lim03}. Exploiting martingales techniques combined with stochastic approximation, Limic and Tarr\`{e}s then settled the conjecture for any $f$ satisfying additional minor technical conditions, later essentially removed by Cotar and Thacker~\cite{CotThac17}.
\smallskip

The effect of branching on ERRWs has not yet been investigated. While stochastic approximation was effective for SERRWs, such continuous-time methods do not seem applicable in presence of branching, which places 
our model somewhere in between LERRWs and SERRWs
due to the population growth. On the other hand, the branching makes a combinatorial analysis, effective for LERRWs, rather difficult. 
This direction requires new ways of implementing discrete-time dynamical systems techniques into probabilistic work. 
\smallskip

Like Sellke and Limic, we consider the triangle graph, although our approach appears to provide key insights for future work on more general graphs, including adding feedback to the model. 
For star graphs, our branching model coincides with certain generalised P\'{o}lya urn models known as \emph{balls and bins models}. 
Notably, those models exhibit neither partial nor full dominance, see~\cite{Giam23} and~\cite{Sid18}. This contrasts with our model on the triangle, where partial dominance emerges as a non-trivial regime occurring with positive probability.

\medskip

\subsection{Heuristics and an overview of the paper}
\label{further}
Denote 
\begin{align*}
\rho_n=\frac{\Vert N_{n}\Vert }{\Vert T_{n}\Vert }
\qquad\text{and}\qquad \rho=\lim_{n\to\infty}\rho_n=1-\mu^{-1},
\end{align*}
see Lemma~\ref{l_rho8}. 
The first key observation is that the dynamics if $(\Theta_n,\pi_n)$ is given by 
\begin{align}
\label{iter_lim_pi}
\pi_{n+1}&=M_{\Theta_n}\pi_n+R_{n+1},\\
\Theta_{n+1}&=(1-\rho)\Theta_n+\rho(\one-\pi_n-\pi_{n+1})+S_{n+1},
\label{iter_lim_th}
\end{align}
consisting of a deterministic dynamical system perturbed by some random fluctuations $(R_{n})$, $(S_{n})$, as shown in 
Proposition~\ref{p_split}. We show in Proposition~\ref{p_error} that the random fluctuations are small, and hence 
the behaviour of $(\Theta_n,\pi_n)$ can be largely inferred from the behaviour of the deterministic dynamical system
\begin{align}
 \label{dynsys0}
  (p_{n+1},q_{n+1})=\Phi(p_n,q_n).
 \end{align}
on $(\Sigma\setminus V)\times \Sigma$ driven by the mapping
 \begin{align*}
 \Phi(p,q)=\big((1-\rho)p+\rho(\one-q-M_p q), M_pq\big).
 \end{align*}

Figure~\ref{fig:twoimages} illustrates the typical dynamics of the sequences $p_n$ 
and $q_n$ 
on the simplex for Scenarios (1) and (3), respectively.
\smallskip

\begin{figure}[h]
    \centering
    \begin{subfigure}[b]{0.48\textwidth}
        \centering
        \includegraphics[width=70mm]{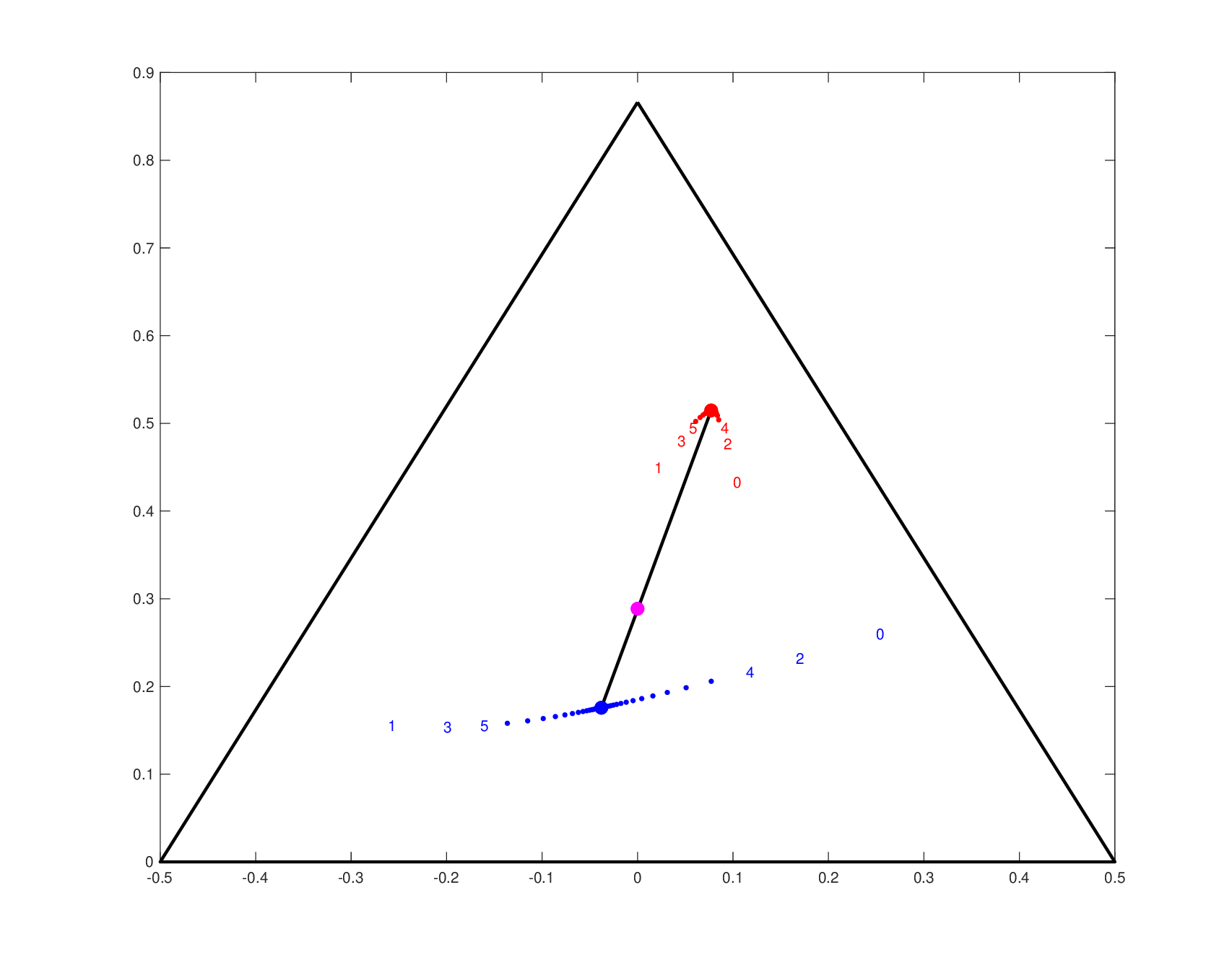}
        \caption{Scenario (1)}
    \end{subfigure}
    \hfill
    \begin{subfigure}[b]{0.48\textwidth}
        \centering
        \includegraphics[width=70mm]{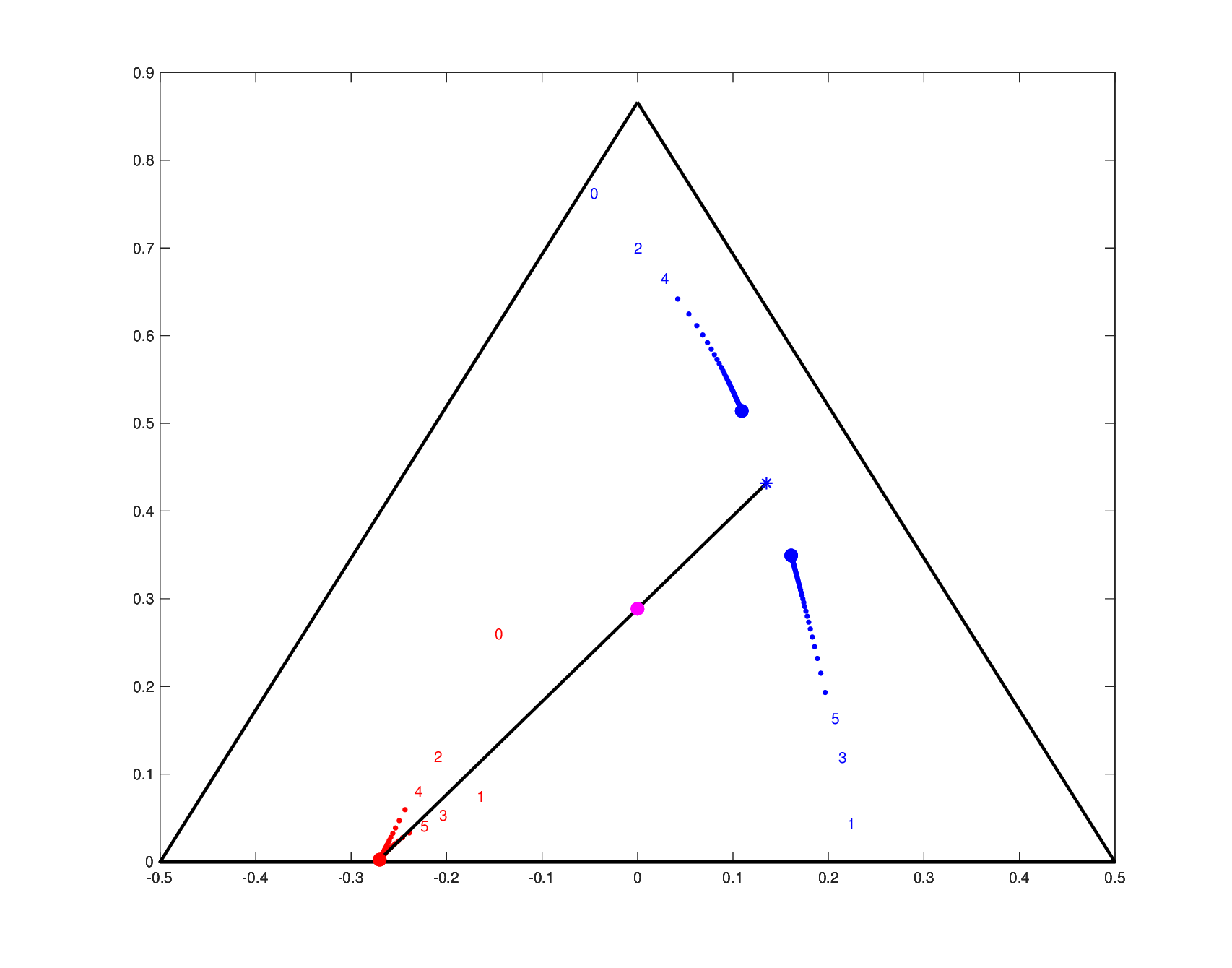}
        \caption{Scenario (3)}
    \end{subfigure}
    \caption{Typical dynamics of $p_n$ (in red) and $q_n$ (in blue) on the simplex (centre pink).}
    \label{fig:twoimages}
\end{figure}
 
Proving convergence in Theorem~\ref{th_conv} is tricky, but once it is established it is not hard to explain the shape of the limits in all five scenarios using the following lemma, which will be proved in Section~\ref{s_edge}.

\begin{lemma} 
\label{la0}
Let
\begin{align*}
\mathcal{L}=\{(u,v,w)\in \R^3: u+v+w=0\}. 
\end{align*}

For any $p\in \Sigma\setminus V$, the space $\mathcal{L}$ is invariant under $M_p$. $M_p$ has an eigenvalue $1$ with an eigenvector $q_{[p]}$ and two further eigenvalues 
$\lambda_{-1}(p),\lambda_0(p)\in [-1,0]$ with normalised eigenvectors $e_{-1}(p), e_0(p)\in \mathcal{L}$. The equalities 
$\lambda_{-1}(p)=-1$ and $\lambda_{0}(p)=0$ occur if and only if $p\in \partial \Sigma\setminus V$. 
\end{lemma}

Now, if $(p_n,q_n)$ converges then the limit point must be a fixed point of $\Phi$ and thus satisfy $q+M_pq=2q_{[p]}$ and 
$M_pq=q$. Since $q\in \Sigma$, by the lemma it needs to be equal to $q_{[p]}$ and hence the limit point must be an 
equilibrium point $(p,q_{[p]})$. 
\smallskip

Further, if $(p_n,q_n)$ is such that $(p_n)$ converges and $(q_n)$ approaches a two-cycle then 
the limit and each of the points in the two-cycle need to satisfy $q+M_pq=2q_{[p]}$ and $M_p^2 q=q$.  By the lemma this means 
that $(M_p+I)(q-q_{[p]})=0$ and $(M_p^2-I)(q-q_{[p]})=0$, where $I$ denotes the identity matrix, implying that $p\in \partial\Sigma\setminus V$ and $q-q_{[p]}=\ell e_{-1}(p)$. 
\smallskip

However, the main challenge in Theorem~\ref{th_conv} is proving convergence. As Lemma~\ref{lemma_l} suggests, a key tool  
for that is the vector
\begin{align*}
v_n=\pi_n-q_{[\Theta_n]}\in\mathcal{L}, \qquad n\in\N.
\end{align*}
In Lemma~\ref{l_uw} we will show that for each $n$
\begin{align}
\label{iter_v}
v_{n+1}
&=\Big[\Big(1-\frac{\rho}{2}\Big)M_{\Theta_n}-\frac{\rho}{2}I\Big]v_n+U_{n+1}\\
\Theta_{n+1}-\Theta_n
&=-\rho(M_{\Theta_n}+I)v_n+W_{n+1}.
\label{iter_incr}
\end{align}
where $(U_n)$ and $(W_n)$ are some negligible random processes.  Further, in Lemma~\ref{l_norm} we show that 
the norm of $M_p$ on $\mathcal{L}$ is bounded by one, being equal to one if and only if $p$ lies on the boundary. 
This leads to the proof of Lemma~\ref{lemma_l} in Section~\ref{s_prelim} and also 
allows us to prove geometric convergence to the equilibrium for orbits that stay away from the boundary in Section~\ref{s_notedge}. 
In order to deal with orbits approaching the boundary, we first study the boundary orbits of the 
unperturbed dynamical system in Section~\ref{s_boundary}. This helps to tame the orbits $(\Theta_n,\pi_n)$
and show that they get sufficiently close to the limits suggested in Theorem~\ref{th_conv}. 
In Section~\ref{s_edge} we perform a delicate asymptotic analysis of the dynamics near the boundary, distinguishing between the cases when 
$\ell>0$ and $\ell=0$. Finally,  those results are combined in the beginning Section~\ref{s_main} to prove Theorem~\ref{th_conv}.
\smallskip 

The main theme of Section~\ref{s_main} is however proving Theorem~\ref{th_pos}. This is done by forcing the system to evolve in a certain, even if unlikely, way that would lead it deeply into the domain of attraction of the desired limiting behaviour, and showing that it will then stay there eventually.  
\smallskip

Finally, the non-monopolistic behaviour is discussed in Section~\ref{s_mon}, where Theorem~\ref{th_mon} is proved via a martingale argument. 
\smallskip

Unless stated otherwise, all arguments involving orbits of the process proceed pointwise almost surely on the sample space. All asymptotics are understood accordingly.

\bigskip


\section{A dynamical system with random perturbations}
\label{s_split}

The aim of this section is to split the random dynamical system $(\Theta_n,\pi_n)$ into a deterministic and a random part, 
to describe the 
corresponding dynamics of $v_n$, and to see how $v_n$ control the increments of $(\Theta_n)$.

\begin{prop} 
\label{p_split}
The process $ (\Theta_n,\pi_n)$ satisfies~\eqref{iter_lim_pi} and ~\eqref{iter_lim_th} with 
\begin{align*}
R_{n+1}^{\ssup i}&=\Upsilon_{n+1}^{\ssup{i\oplus 1}}\bar P_n^{\ssup{i\oplus 1}}
+\Upsilon_{n+1}^{\ssup{i\ominus 1}}P_n^{\ssup{i\ominus 1}}+\Xi_{n+1}^{\ssup i },\\
S_{n+1}^{\ssup i}&=\rho_{n+1}\big[\Upsilon_{n+1}^{\ssup{i\oplus 1}}+\Upsilon_{n+1}^{\ssup{i\ominus 1}}\big]
+(\rho_{n+1}-\rho)\big[1-\pi_n^{\ssup i}-\pi_{n+1}^{\ssup i}-\Theta_n^{\ssup i}\big],
\end{align*}
where 
\begin{align}
\label{def_r}
\Upsilon_{n+1}^{\ssup i}=\frac{1}{\Vert N_{n+1}\Vert }\sum_{k=1}^{N_n^{\ssup i}}Z^{\ssup i}_{n+1,k}-\pi_n^{\ssup i}
\qquad\text{and}\qquad
\Xi_{n+1}^{\ssup i}=
\frac{C_{n+1}^{\ssup{i\ominus 1}}-C_{n+1}^{\ssup{i\oplus 1}}}{\Vert N_{n+1}\Vert },
\end{align}
with 
\begin{align}
\label{cb}
C_{n+1}^{\ssup i}=B_{n+1}^{\ssup i}-P_n^{\ssup i}\sum_{k=1}^{N_n^{\ssup i}}Z^{\ssup i}_{n+1,k}
\end{align}
denoting the centralised binomials. 

\end{prop}

\begin{proof}
Observe that~\eqref{def_t} and~\eqref{def_n} can be rewritten as
\begin{align}
\label{def_t_c}
T^{\ssup i}_{n+1}&=T_n^{\ssup i}+
P_n^{\ssup{i\oplus 1}}\sum_{k=1}^{N_n^{\ssup{i\oplus 1}}}Z^{\ssup{i\oplus 1}}_{n+1,k}
+\bar P_n^{\ssup{i\ominus 1}}\sum_{k=1}^{N_n^{\ssup{i\ominus 1}}}Z^{\ssup{i\ominus 1}}_{n+1,k}
+C_{n+1}^{\ssup{i\oplus 1}}-C_{n+1}^{\ssup{i\ominus 1}},\\
N^{\ssup i}_{n+1}&=
P_n^{\ssup{i\ominus 1}}\sum_{k=1}^{N_n^{\ssup{i\ominus 1}}}Z^{\ssup{i\ominus 1}}_{n+1,k}
+\bar P_n^{\ssup{i\oplus 1}}\sum_{k=1}^{N_n^{\ssup{i\oplus 1}}}Z^{\ssup{i\oplus 1}}_{n+1,k}
+C_{n+1}^{\ssup{i\ominus 1}}-C_{n+1}^{\ssup{i\oplus 1}}.
\label{def_n_c}
\end{align}
Normalising~\eqref{def_t_c} and~\eqref{def_n_c} as in~\eqref{def_thpi} and using $\frac{\Vert T_n\Vert }{\Vert T_{n+1}\Vert }=1-\rho_{n+1}$, 
we obtain
\begin{align}
\Theta^{\ssup i}_{n+1}&=(1-\rho_{n+1})\Theta_n^{\ssup i}+
\rho_{n+1}\big[\pi_n^{\ssup{i\oplus 1}}P_n^{\ssup{i\oplus 1}}
+\pi_n^{\ssup{i\ominus 1}}\bar P_n^{\ssup{i\ominus 1}}
+\Upsilon_{n+1}^{\ssup{i\oplus 1}}P_n^{\ssup{i\oplus 1}}
+\Upsilon_{n+1}^{\ssup{i\ominus 1}}\bar P_n^{\ssup{i\ominus 1}}
-\Xi^{\ssup i}_{n+1}\big],
\label{iter_th2}\\
\pi^{\ssup i}_{n+1}&=
\pi_n^{\ssup{i\oplus 1}}\bar P_n^{\ssup{i\oplus 1}}
+\pi_n^{\ssup{i\ominus 1}}P_n^{\ssup{i\ominus 1}}
+\Upsilon_{n+1}^{\ssup{i\oplus 1}}\bar P_n^{\ssup{i\oplus 1}}
+\Upsilon_{n+1}^{\ssup{i\ominus 1}}P_n^{\ssup{i\ominus 1}}+\Xi_{n+1}^{\ssup i }. 
\label{iter_pi2}
\end{align}
Using the fact that $\pi_n\in \Sigma$ we can rewrite~\eqref{iter_th2} and~\eqref{iter_pi2} as
\begin{align}
\label{iter_pi3}
\pi^{\ssup i}_{n+1}&
= \pi_n^{\ssup{i\oplus 1}}\frac{\Theta_n^{\ssup{i\ominus 1}}}{\Theta_n^{\ssup{i}}+\Theta_n^{\ssup{i\ominus 1}}}
+\pi_n^{\ssup{i\ominus 1}}\frac{\Theta_n^{\ssup{i\oplus 1}}}{\Theta_n^{\ssup{i\oplus 1}}+\Theta_n^{\ssup i}}+R^{\ssup i}_{n+1},\\
\Theta^{\ssup i}_{n+1}
&=(1-\rho)\Theta_n^{\ssup i}
+\rho\big[1-\pi_n^{\ssup i}-\pi_{n+1}^{\ssup i}\big]+S^{\ssup i}_{n+1},
\label{iter_th3}
\end{align} 
which is equivalent to~\eqref{iter_lim_pi} and ~\eqref{iter_lim_th}.
\end{proof}

\begin{lemma}  
\label{l_uw}
The iterations~\eqref{iter_v} and~\eqref{iter_incr} hold with 
\begin{align*}
U_{n+1}=\Big(1-\frac{\rho}{2}\Big)R_{n+1}+\frac 1 2 S_{n+1} 
\qquad\text{ and }\qquad W_{n+1}=S_{n+1}-\rho R_{n+1}.
\end{align*}
\end{lemma}

\begin{proof} Using~\eqref{qp}, \eqref{iter_lim_pi}, \eqref{iter_lim_th} and $M_{\Theta_n}q_{[\Theta_n]}=q_{[\Theta_n]}$  we get for each $n$
\begin{align*}
v_{n+1}
&=\pi_{n+1}-\frac 1 2\Big[ \one-(1-\rho)\Theta_n-\rho(\one-\pi_n-\pi_{n+1})-S_{n+1}\Big]\notag\\
&=\Big(1-\frac{\rho}{2}\Big)\pi_{n+1}-(1-\rho)q_{[\Theta_n]}-\frac{\rho}{2}\pi_n+\frac{1}{2}S_{n+1}\notag\\
&=\Big(1-\frac{\rho}{2}\Big)M_{\Theta_n}\pi_{n}-(1-\rho)q_{[\Theta_n]}-\frac{\rho}{2}\pi_n+U_{n+1}
=\Big[\Big(1-\frac{\rho}{2}\Big)M_{\Theta_n}-\frac{\rho}{2}I\Big]v_n+U_{n+1}
\end{align*}
and
\begin{align*}
\Theta_{n+1}-\Theta_n
&=\rho(\one-\Theta_n-\pi_n-\pi_{n+1})+S_{n+1}\notag\\
&=\rho(2q_{[\Theta_n]}-\pi_n-M_{\Theta_n}\pi_n)+W_{n+1}
=-\rho(M_{\Theta_n}+I)v_n+W_{n+1},
\end{align*}
as required. 
\end{proof}


\section{Random perturbations}

The aim of this section is to show that the random perturbations are negligible. 
\smallskip

We will begin by showing that the Galton-Watson branching mechanism
behaves similarly enough to the deterministic branching with rate $\mu$.   

\begin{lemma}  
\label{l_n8}
The process $\mu^{-n}\Vert N_n\Vert$ converges almost surely to a positive random variable and, in particular, for any $\nu\in (1,\mu)$, $\Vert N_n\Vert \ge \nu^{n}$ eventually almost surely.  
Furthermore, for any $\nu\in (1,\mu)$,
\begin{align*}
\1_{\{\Vert N_m\Vert \ge \nu^m\}}\P_{\mathcal{F}_m}\Big(\exists n>m\text{ such that }\Vert N_n\Vert<\nu^n\Big)\to 0
\end{align*}
as $m\to\infty$ uniformly on $\Omega$. 
\end{lemma}

\begin{proof} We have
\begin{align*}
\Vert N_n\Vert &=\sum_{i=1}^{\Vert N_0\Vert }X^{\sssup i}_{n},
\end{align*}
where $(X^{\sssup i}_{n})_{n\ge 0}, 0\le i\le \Vert N_0\Vert $ are independent Galton-Watson processes with the offspring distribution $Z$. 
To prove the first statement, it remains to observe that each
$\mu^{-n}X^{\sssup i}_n$ converges almost surely to a positive random variable by the $L^2$ Martingale Convergence Theorem, see~\cite[\S 1.6]{AthNey70}. 
\smallskip

To prove the second statement, observe that on the event $\{\Vert N_m\Vert \ge \nu^m\}$ we have 
\begin{align}
\P_{\mathcal{F}_m}\Big(\exists n>m\text{ such that }\Vert  N_n\Vert <\nu^n\Big)
&\le \sum_{n=m}^{\infty}\P_{\mathcal{F}_m}\Big(\Vert N_{n+1}\Vert <\nu^{n+1} \text{ and } \Vert N_n\Vert \ge \nu^n\Big)\notag\\
&= \sum_{n=m}^{\infty}\E_{\mathcal{F}_m}\Big[\1_{\{\Vert N_n\Vert \ge \nu^n\}}
\P_{\mathcal{F}_n}\big(\Vert N_{n+1}\Vert <\nu^{n+1} \big)
\Big].
\label{mmnn}
\end{align}
Further, on the event $\{\Vert N_n\Vert \ge \nu^n\}$ we have by Chebychev's inequality
\begin{align*}
\P_{\mathcal{F}_n}\big(\Vert N_{n+1}\Vert <\nu^{n+1} \big)\le 
\P_{\mathcal{F}_n}\Big(\big| \Vert N_{n+1}\Vert -\mu \Vert N_n\Vert \big|\ge (\mu-\nu) \Vert N_n\Vert  \Big)\le 
\frac{\text{var(Z)}}{(\mu-\nu)^2\Vert N_n\Vert }\le \frac{\text{var(Z)}}{(\mu-\nu)^2\nu^n}.
\end{align*} 
Combining this with~\eqref{mmnn} we obtain 
\begin{align*}
\1_{\{\Vert N_m\Vert \ge \nu^m\}}\P_{\mathcal{F}_m}\Big(\exists n>m\text{ such that }\Vert N_n\Vert <\nu^n\Big)
\le \frac{\text{var(Z)}}{(\mu-\nu)^2}\sum_{n=m}^{\infty}\nu^{-n}\to 0
\end{align*}
as $m\to\infty$ uniformly on $\Omega$. 
\end{proof}

The following lemma shows that if the population is big enough 
it will increase exponentially close to 
a population deterministically growing at rate $\mu$. 

\begin{lemma} 
\label{l_n11}
For any $1<\nu<\hat\nu <\mu^{\frac 1 2}$
\begin{align*}
\1_{\{\Vert N_m\Vert \ge \hat\nu^{2m}\}}\P_{\mathcal{F}_m}\Big(\exists n\ge m\text{ \rm such that }
\Big|\frac{\mu\Vert N_{n}\Vert }{\Vert N_{n+1}\Vert }-1\Big|>\nu^{-n}\Big)
& \to 0
\end{align*}
as $m\to\infty$ uniformly on $\Omega$. In particular, 
$\frac{\Vert N_{n+1}\Vert }{\Vert N_n\Vert }=\mu+O(\nu^{-n})$ almost surely.
\end{lemma}

\begin{proof} 
We will prove a statement equivalent
to the first one, namely that for any $1<\nu<\hat\nu <\mu^{\frac 1 2}$ 
\begin{align*}
\1_{\{\Vert N_m\Vert \ge \hat\nu^{2m}\}}\P_{\mathcal{F}_m}\Big(\exists n\ge m\text{ \rm such that }
\Big|\frac{\Vert N_{n+1}\Vert }{\Vert N_n\Vert }-\mu\Big|>\nu^{-n}\Big)
& \to 0
\end{align*}
as $m\to\infty$ uniformly on $\Omega$.
Observe that
on the event $\{\Vert N_m\Vert \ge \hat\nu^{2m}\}$ we have 
\begin{align}
&\P_{\mathcal{F}_m}\Big(\exists n\ge m\text{ such that }
\Big|\frac{\Vert N_{n+1}\Vert }{\Vert N_n\Vert }-\mu\Big|>\nu^{-n}\Big)\notag\\
&\le \sum_{n=m}^{\infty}\P_{\mathcal{F}_m}\Big(
\Big|\frac{\Vert N_{n+1}\Vert }{\Vert N_n\Vert }-\mu\Big|>\nu^{-n}, \Vert N_n\Vert \ge\hat \nu^{2n}\Big)
+ \P_{\mathcal{F}_m}\Big(\exists n>m\text{ such that }\Vert N_n\Vert<\hat\nu^{2n}\Big)\notag\\
&= \sum_{n=m}^{\infty}\E_{\mathcal{F}_m}\Big[\1_{\{\Vert N_n\Vert \ge \hat\nu^{2n}\}}
\P_{\mathcal{F}_n}
\Big(\Big|\frac{\Vert N_{n+1}\Vert }{\Vert N_n\Vert }-\mu\Big|>\nu^{-n}\Big)\Big]
+ \P_{\mathcal{F}_m}\Big(\exists n>m\text{ such that }\Vert N_n\Vert<\hat\nu^{2n}\Big). 
\label{n7}
\end{align}
On the event $\{\Vert N_n\Vert \ge \hat\nu^{2n}\}$ we have by Chebychev's inequality 
\begin{align*}
\P_{\mathcal{F}_n}\Big(\Big|\frac{\Vert N_{n+1}\Vert }{\Vert N_n\Vert }-\mu\Big|>\nu^{-n}\Big)
\le \frac{\nu^{2n}\text{var}(Z)}{\Vert N_n\Vert}\le \Big(\frac{\nu}{\hat \nu}\Big)^{2n}\text{var}(Z).
\end{align*}
Combining this with~\eqref{n7} and Lemma~\ref{l_n8} we obtain  
\begin{align*}
 &\1_{\{\Vert N_m\Vert \ge \hat\nu^{2m}\}}\P_{\mathcal{F}_m}\Big(\exists n\ge m\text{ such that }
\Big|\frac{\Vert N_{n+1}\Vert }{\Vert N_n\Vert }-\mu\Big|>\nu^{-n}\Big)\\
&\le \text{var}(Z)\sum_{n=m}^{\infty}\Big(\frac{\nu}{\hat \nu}\Big)^{2n}
+\1_{\{\Vert N_m\Vert \ge \hat\nu^{2m}\}}\P_{\mathcal{F}_m}\Big(\exists n>m\text{ such that }\Vert N_n\Vert<\hat\nu^{2n}\Big)\to 0
\end{align*}
as $m\to\infty$ uniformly on $\Omega$. 
\smallskip

To prove the second statement, it remains to observe that 
\begin{align*}
\P\Big(\Big|\frac{\Vert N_{n+1}\Vert }{\Vert N_n\Vert }-\mu\Big|>\nu^{-n} \text{ i.o.}\Big)
&\le \lim_{m\to\infty}\E\Big[\1_{\{\Vert N_m\Vert \ge \hat\nu^{2m}\}}\P_{\mathcal{F}_m}\Big(\exists n\ge m\text{ such that }\Big|\frac{\Vert N_{n+1}\Vert }{\Vert N_n\Vert }-\mu\Big|>\nu^{-n}\Big)\Big]\\
&+\lim_{m\to\infty}\P\big(\Vert N_m\Vert < \hat\nu^{2m}\big)=0
\end{align*}
by the previous statement and Lemma~\ref{l_n8}.
\end{proof}

The following lemma shows that if the population is big enough, 
then its spacial fluctuations are exponentially small. 

\begin{lemma} 
\label{l_xi}
For any $1<\nu<\hat\nu <\mu^{\frac 1 2}$
\begin{align*}
\1_{\{\Vert N_m\Vert \ge \hat\nu^{2m}\}}\P_{\mathcal{F}_m}\Big(\exists n\ge m\text{ such that }\Vert \Xi_{n+1}\Vert>\nu^{-n}\Big)\to 0
\end{align*}
as $m\to\infty$ uniformly on $\Omega$. In particular, $\Vert \Xi_{n+1}\Vert=O(\nu^{-n})$ almost surely. 
\end{lemma}

\begin{proof} 
The proof is similar to that of Lemma~\ref{l_n11}, and it suffices 
to prove the first half of the claim
coordinatewise. For each $i$, on the event 
$\{\Vert N_m\Vert \ge \hat\nu^{2m}\}$ we have 
\begin{align}
& \P_{\mathcal{F}_m}\Big(\exists n\ge m\text{ such that }
|\Xi^{\ssup i}_{n+1}|>\nu^{-n}\Big)\notag\\
&\le\sum_{n=m}^{\infty}\P_{\mathcal{F}_m}\Big(
\frac{\Vert N_{n+1}\Vert }{\Vert N_n\Vert }|\Xi^{\ssup i}_{n+1}|>\nu^{-n}, \Vert N_n\Vert \ge\hat \nu^{2n}\Big)+\P_{\mathcal{F}_m}\Big(\exists n>m\text{ such that }\Vert N_n\Vert<\hat\nu^{2n}\Big)\notag\\
&= \sum_{n=m}^{\infty}\E_{\mathcal{F}_m}\Big[\1_{\{\Vert N_n\Vert\ge \hat \nu^{2n}\}}
\P_{\mathcal{F}_n}\Big(
\frac{|C_{n+1}^{\ssup{i\ominus 1}}-C_{n+1}^{\ssup{i\oplus 1}}|}{\Vert N_n\Vert }>\nu^{-n}\Big)\Big]+\P_{\mathcal{F}_m}\Big(\exists n>m\text{ such that }\Vert N_n\Vert<\hat\nu^{2n}\Big).
\label{n9}
\end{align}
On the event $\{\Vert N_n\Vert \ge \hat\nu^{2n}\}$ we have by Chebychev's inequality
\begin{align*}
\P_{\mathcal{F}_n}\Big(
\frac{|C_{n+1}^{\ssup{i\ominus 1}}-C_{n+1}^{\ssup{i\oplus 1}}|}{\Vert N_n\Vert }>\nu^{-n}\Big)
\le 2\nu^{2n}\frac{\mu[N_{n}^{\ssup{i\oplus 1}}+N_n^{\ssup{i\ominus 1}}]}{\Vert N_n\Vert^2}
\le \frac{2\mu\nu^{2n}}{\Vert N_n\Vert}\le 2\mu\Big(\frac{\nu}{\hat \nu}\Big)^{2n}.
\end{align*}
Combining this with~\eqref{n9} and Lemma~\ref{l_n8} we obtain  
\begin{align*}
&\1_{\{\Vert N_m\Vert \ge \hat\nu^{2m}\}}\P_{\mathcal{F}_m}\Big(\exists n\ge m\text{ such that }
|\Xi^{\ssup i}_{n+1}|>\nu^{-n}\Big)\notag\\
&\le 2\mu \sum_{n=m}^{\infty}\Big(\frac{\nu}{\hat \nu}\Big)^{2n}
+\1_{\{\Vert N_m\Vert \ge \hat\nu^{2m}\}}\P_{\mathcal{F}_m}\Big(\exists n>m\text{ such that }\Vert N_n\Vert<\hat\nu^{2n}\Big)\to 0
\end{align*}
as $m\to\infty$ uniformly on $\Omega$. 
\smallskip

The second statement follows from the first one in the same way as it is done in Lemma~\ref{l_n11}.
\end{proof}

Our next aim is to control the branching- and movement-dependent process $(\Upsilon_{n+1})$ by a processes which only depends on the branching.
Let 
\begin{align*}
\Gamma_{n+1}^{\ssup i}
=\frac{1}{\Vert N_{n}\Vert}\max_{1\le j\le \Vert N_n\Vert}
\Big|\sum_{k=1}^{j}
(Z^{\ssup i}_{n+1,k}-\mu)\Big|.
\end{align*}
Observe that 
\begin{align*}
|\Upsilon_{n+1}^{\ssup i}|
=\Big|\frac{1}{\Vert N_{n+1}\Vert }\sum_{k=1}^{N_n^{\ssup i}}(Z^{\ssup i}_{n+1,k}-\mu)
+\Big(\frac{\mu\Vert N_n\Vert}{\Vert N_{n+1}\Vert }-1\Big)\pi_n^{\ssup i}\Big|
\le \frac{\Vert N_n\Vert}{\Vert N_{n+1}\Vert }\Gamma_{n+1}^{\ssup i}
+\Big|\frac{\mu\Vert N_n\Vert}{\Vert N_{n+1}\Vert }-1\Big|\pi_n^{\ssup i}
\end{align*}
implying 
\begin{align}
\label{upga}
\Vert \Upsilon_{n+1}\Vert \le \Vert \Gamma_{n+1}\Vert + \Big|\frac{\mu\Vert N_n\Vert}{\Vert N_{n+1}\Vert }-1\Big|.
\end{align}
The following lemma shows that if the population is big enough, 
then its branching fluctuations are exponentially small.

\begin{lemma} 
\label{l_ga}
For any $1<\nu<\hat\nu <\mu^{\frac 1 2}$
\begin{align*}
\1_{\{\Vert N_m\Vert \ge \hat\nu^{2m}\}}\P_{\mathcal{F}_m}\Big(\exists n\ge m\text{ \rm such that }
\Vert\Gamma_{n+1}\Vert >\nu^{-n}\Big)
& \to 0
\end{align*}
as $m\to\infty$ uniformly on $\Omega$. In particular, 
$\Vert\Gamma_{n+1}\Vert=O(\nu^{-n})$ almost surely.
\end{lemma}

\begin{proof} 
The proof is similar to that of Lemma~\ref{l_n11}, and it suffices to prove the first half of the claim coordinatewise. For each 
$i$, on the event $\{\Vert N_m\Vert \ge \hat\nu^{2m}\}$ we have 
\begin{align}
\label{kk1}
&\P_{\mathcal{F}_m}\Big(\exists n\ge m\text{ such that }
\Gamma_{n+1}^{\ssup i}>\nu^{-n}\Big)\notag\\
&\le \sum_{n=m}^{\infty}\E_{\mathcal{F}_m}\Big[\1_{\{\Vert N_n\Vert \ge \hat\nu^{2n}\}}
\P_{\mathcal{F}_n}
\big(\Gamma_{n+1}^{\ssup i}>\nu^{-n}\big)\Big]
+ \P_{\mathcal{F}_m}\Big(\exists n>m\text{ such that }\Vert N_n\Vert<\hat\nu^{2n}\Big). 
\end{align}
On the event $\{\Vert N_n\Vert \ge \hat\nu^{2n}\}$ we have by Doob's submartingale inequality
\begin{align*}
\P_{\mathcal{F}_n}
\big(\Gamma_{n+1}^{\ssup i}>\nu^{-n}\big)\le \frac{\nu^{2n}}{\Vert N_n\Vert}
\text{var}(Z)\le \Big(\frac{\nu}{\hat \nu}\Big)^{2n}\text{var}(Z).
\end{align*}
Combining this with~\eqref{kk1} and Lemma~\ref{l_n8} we obtain  
\begin{align*}
 &\1_{\{\Vert N_m\Vert \ge \hat\nu^{2m}\}}\P_{\mathcal{F}_m}\Big(\exists n\ge m\text{ such that }
\Gamma_{n+1}^{\ssup i}>\nu^{-n}\Big)\\
&\le \text{var}(Z) \sum_{n=m}^{\infty}\Big(\frac{\nu}{\hat \nu}\Big)^{2n}
+\1_{\{\Vert N_m\Vert \ge \hat\nu^{2m}\}}\P_{\mathcal{F}_m}\Big(\exists n>m\text{ such that }\Vert N_n\Vert<\hat\nu^{2n}\Big)
\to 0
\end{align*}
as $m\to\infty$ uniformly on $\Omega$, which is equivalent to the first statement.
\smallskip

The second statement follows from the first one in the same way as it 
is done in Lemma~\ref{l_n11}. 
\end{proof}

In this section and, more importantly, in Sections~\ref{s_notedge} and~\ref{s_main}, we will need the following quantitative 
estimate for sequences that decay exponentially for a while.  

\begin{lemma}
\label{lemma_notgeom}
Let $c\in (0,1), a>0$ and let $(u_n)$ be a positive sequence such that $u_{n+1}\le c u_n +a\nu^{-n}$ for all $m\le n<k\le\infty$. Then, 
for any $\gamma\in (\max\{c, \frac{1}{\nu}\}, 1) $
\begin{align*}
u_{n} 
\le \gamma^{n-m} \Big[u_m+\nu^{-m}\frac{a\nu}{\gamma\nu-1}\Big], \qquad m\le n<k.
\end{align*} 
In particular, if $k=\infty$ then 
$(u_n)$ exponentially decays to zero. 
\end{lemma}

\begin{proof} 
Since $c\le \gamma$ we have $u_{n+1}\le \gamma u_n +a\nu^{-n}$ for all $m\le n<k$. Iterating, we obtain 
\begin{align*}
u_{n} 
\le \gamma^{n-m} u_m+a\sum_{i=0}^{n-m-1}\gamma^{i}\nu^{-n+i+1},
\end{align*}
and the statement now follows by computing the geometric sum and using $\gamma\nu>1$. 
\end{proof}

In the following lemma we show that the ratio between the number of traversals (that is, of the total number of particles in all generations) 
and the number of particles in the most recent generation for the Galton-Watson branching behaves in the same way as for the deterministic 
branching at rate $\mu$. 
This is similar to Lemmas~\ref{l_n11} and~\ref{l_ga} with the 
exception that there is a delay between when the population becomes large 
and when the behaviour of $\rho_n$ becomes nearly deterministic, which is due to $\rho_n$ depending on all previous generations.

\begin{lemma} 
\label{l_rho8}
For any $1<\nu<\hat\nu <\mu^{\frac 1 2}$
there is $\varpi:\N\to \N$ such that $\varpi(m)\ge m$ for all $m$ and 
\begin{align*}
\1_{\{\Vert N_m\Vert \ge \hat\nu^{2m}\}}\P_{\mathcal{F}_m}\Big(\exists n\ge \varpi(m)\text{ \rm such that }
|\rho_{n+1}^{-1}-\rho^{-1}| >\nu^{-n}\Big)
& \to 0
\end{align*}
as $m\to\infty$ uniformly on $\Omega$. In particular, 
$\rho_{n+1}=\rho+O(\nu^{-n})$ almost surely. 
Furthermore, 
\begin{align*}
\1_{\{\Vert N_m\Vert \ge \hat\nu^{2m}, |\rho_m^{-1}-\rho^{-1}|\le \nu^{-(m-1)}\}}\P_{\mathcal{F}_m}\Big(\exists n\ge m\text{ \rm such that }
|\rho_{n+1}^{-1}-\rho^{-1}| >\nu^{-n}\Big)
& \to 0
\end{align*}
as $m\to\infty$ uniformly on $\Omega$. 
\end{lemma}

\begin{proof} 
We assume that $m$ is sufficiently large. 
For all $n\in\N$ we have
\begin{align*}
\rho_{n+1}^{-1}=\frac{\Vert T_{n+1}\Vert }{\Vert N_{n+1}\Vert }
=\frac{\Vert N_{n+1}\Vert + \Vert T_n\Vert}{\Vert N_{n+1}\Vert}
=1+\rho_{n}^{-1}\frac{\Vert N_{n}\Vert }{\Vert N_{n+1}\Vert }.
\end{align*}
Substituting $u_n= \rho_n^{-1}-\rho^{-1}$ we obtain 
\begin{align*}
u_{n+1}
=u_{n}\frac{\Vert N_{n}\Vert }{\Vert N_{n+1}\Vert }
+\frac{1}{\mu-1}\Big(\frac{\mu\Vert N_{n}\Vert }{\Vert N_{n+1}\Vert }-1
\Big).
\end{align*}
Let $\bar\nu\in (\nu,\hat\nu)$ and $c\in (\frac 1 \mu, \frac{1}{\bar \nu})$. Consider the event 
\begin{align*}
\mathcal{E}_m=\Big\{\frac{\Vert N_{n}\Vert }{\Vert N_{n+1}\Vert }\le c
\text{ and }\Big|\frac{\mu\Vert N_{n}\Vert }{\Vert N_{n+1}\Vert }-1
\Big|<\bar\nu^{-n}
\text{ for all }n\ge m\Big\}.
\end{align*}
On the event $\mathcal{E}_m$ we have for all $n\ge m$
\begin{align*}
|u_{n+1}|\le c|u_n|+\frac{\bar\nu^{-n}}{\mu-1}.
\end{align*}
By Lemma~\ref{lemma_notgeom} with some $\gamma\in (\frac{1}{\bar\nu}, \frac 1 \nu)$ we obtain 
\begin{align}
\label{ooo1}
|u_{n+1}|\le \gamma^{n+1-m}\Big[|u_m|+\bar\nu^{-m}\frac{\bar\nu}{(\mu-1)(\gamma\bar\nu-1)}\Big].
\end{align}
If we have no information about $u_m$ we can simply estimate it as $|u_m|\le 2m$, 
which leads to $|u_{n+1}|\le 3m\gamma^{n-m}$. 
Since $\gamma\nu<1$, for each $m$ we can choose $\varpi(m)$ sufficiently 
large to ensure $3m(\gamma\nu)^{\varpi(m)}\gamma^{-m}\le 1$. This implies 
for all $n\ge\varpi(m)$
\begin{align*}
|u_{n+1}|
\le 3m(\gamma\nu)^n\gamma^{-m}\nu^{-n}
\le 3m(\nu\gamma)^{\varpi(m)}\gamma^{-m}\nu^{-n}\le \nu^{-n}.
\end{align*}
on the event $\mathcal{E}_m$. This implies 
\begin{align*}
\1_{\{\Vert N_m\Vert \ge \hat\nu^{2m}\}}\P_{\mathcal{F}_m}\Big(\exists n\ge \varpi(m)\text{ \rm such that }
|u_{n+1}| >\nu^{-n}\Big)
\le \1_{\{\Vert N_m\Vert \ge \hat\nu^{2m}\}}\P_{\mathcal{F}_m}(\mathcal{E}_m^c)\to 0
\end{align*}
uniformly on $\Omega$ by Lemma~\ref{l_n11}. 
\smallskip

If we know that $|u_m|\le \nu^{-(m-1)}$ then~\eqref{ooo1} together with 
$\gamma\nu<1$ and $\nu<\bar\nu$ imply 
for all $n\ge m$
\begin{align*}
|u_{n+1}|\le \Big[(\gamma\nu)^{n+1-m}+(\gamma\nu)^n(\gamma\bar\nu)^{-m}\frac{\gamma\bar\nu}{(\mu-1)(\gamma\bar\nu-1)}\Big]\nu^{-n}
\le \Big[\gamma\nu + \Big(\frac{\nu}{\bar\nu}\Big)^{m}
\frac{\gamma\bar\nu}{(\mu-1)(\gamma\bar\nu-1)}\Big]\nu^{-n}
\le \nu^{-n}
\end{align*}
on the event $\mathcal{E}_m\cap \big\{|u_m|\le \nu^{-(m-1)}\big\}$.
This implies 
\begin{align*}
\1_{\{\Vert N_m\Vert \ge \hat\nu^{2m}, |u_m|\le \nu^{-(m-1)}\}}\P_{\mathcal{F}_m}\Big(\exists n\ge m\text{ \rm such that }
|u_{n+1}| >\nu^{-n}\Big)
& \le 
\1_{\{\Vert N_m\Vert \ge \hat\nu^{2m}\}}\P_{\mathcal{F}_m}(\mathcal{E}_m^c)\to 0
\end{align*}
uniformly on $\Omega$ by Lemma~\ref{l_n11}.
\smallskip

Finally, $\rho_{n+1}=\rho+O(\nu^{-n})$ follows from the
first statement similarly to Lemma~\ref{l_n11}.
\end{proof}

\begin{prop} 
\label{p_error}
For any $\nu\in (1,\mu^{\frac 1 2})$, $\Vert R_{n+1}\Vert,  \Vert S_{n+1}\Vert, \Vert U_{n+1}\Vert, \Vert W_{n+1}\Vert= O(\nu^{-n})$ almost surely.
\end{prop}

\begin{proof} This follows immediately from~\eqref{upga} and Lemmas~\ref{l_uw}, ~\ref{l_n11}, ~\ref{l_xi}, \ref{l_ga} and \ref{l_rho8}.
\end{proof}

\bigskip


\section{Distance from the equilibrium}
\label{s_prelim}

Observe that $v_n=\pi_n-q_{[\Theta_n]}$ can be seen as a measure of how much the sequence $(\Theta_n,\pi_n)$ is out of the equilibrium state. In this section we 
will prove Lemma~\ref{lemma_l} by showing that a norm of the 
deterministic transformation from $v_n$ to $v_{n+1}$ is bounded by one, and the random perturbations are negligible.
\smallskip

Recall that $v_n\in \mathcal{L}$, which is an invariant plane under $M_p$. For any three-dimensional square matrix $A$ having $\mathcal{L}$ as an invariant plane, we denote by $\vertiii{A}$ the norm of the restriction of the operator $A$ to the plane $\mathcal{L}$
generated by the $\ell_1$-norm on $\R^3$, that is, 
$$\vertiii{A}=\sup_{r\in \mathcal{L}: \Vert r\Vert =1}\Vert Ar\Vert .$$

For any $p=(x,y,z)\in\Sigma\setminus V$, we denote 
\begin{align*}
\chi(p)=\min\Big\{\frac{x}{x+y}, \frac{y}{x+y}, \frac{x}{x+z}, \frac{z}{x+z}, \frac{y}{y+z}, \frac{z}{y+z}\Big\}.
\end{align*}

\begin{lemma} 
\label{l_norm}
$\vertiii{M_p}=1-\chi(p)$ for any $p\in\Sigma\setminus V$. 
\end{lemma}

\begin{proof} 
It is easy to see that the intersection of $\mathcal{L}$ with the set $\{v\in \R^3: \Vert v\Vert =1\}$ is the regular hexagon 
with the vertices $V_{\text{hex}}=\{(\frac 1 2, -\frac 1 2, 0), (-\frac 1 2, \frac 1 2, 0), (\frac 1 2, 0, -\frac 1 2), (-\frac 1 2, 0, \frac 1 2),
(0, \frac 1 2, -\frac 1 2), (0, -\frac 1 2, \frac 1 2)\}$.  By convexity of the norm
\begin{align}
\label{hex}
\vertiii{M_{p}}=\sup_{v\in V_{\text{hex}}}\Vert M_{p}v\Vert .
\end{align}
Denote $p=(x,y,z)$. For  $v=(\frac 1 2, -\frac 1 2, 0)$ we have 
$$
\Vert M_{p} v\Vert 
= 1-\min\Big\{\frac{x}{x+z},\frac{y}{y+z}\Big\}, 
$$
and the statement  now follows from~\eqref{hex} by symmetry.
\end{proof}
\smallskip


\begin{proof}[Proof of Lemma~\ref{lemma_l}] 
Let $\nu\in (1,\mu^{\frac 1 2})$. By Proposition~\ref{p_error} we can choose $m$ sufficiently large so that $\Vert U_{n+1}\Vert \le \nu^{-n}$ for all $n\ge m$. 
Denote
\begin{align*}
P_{n,m}=\prod_{i=1}^{n-m}\Big[\Big(1-\frac{\rho}{2}\Big)M_{\Theta_{n-i}}-\frac{\rho}{2}I\Big].
\end{align*}
Iterating~\eqref{iter_v} we obtain for all $n>m$
\begin{align}
\label{ab}
v_n=P_{n,m}v_m+\sum_{i=m+1}^n P_{n,i}U_i. 
\end{align}
Denote $\hat v_n=P_{n,m}v_m$ for all $n>m$. 
It follows from Lemma~\ref{l_norm} that
$\vertiii{(1-\frac{\rho}{2})M_{\Theta_{n}}-\frac{\rho}{2}I}\le 1$ and hence 
$\Vert \hat v_{n+1}\Vert \le \Vert \hat v_{n}\Vert $ for all $n$
implying convergence of the sequence $\Vert \hat v_n\Vert $. Lemma~\ref{l_norm} also yields that $\vertiii{P_{n,i}}\le 1$, and hence 
by~\eqref{ab} for any $n, k> m$ we have 
\begin{align*}
&\big\Vert |v_n\Vert -\Vert v_k\Vert \big|\le \Vert v_n-\hat v_n\Vert +\big\Vert |\hat v_n\Vert -\Vert \hat v_k\Vert \big|+\Vert \hat v_k-v_k\Vert 
\le \big\Vert |\hat v_n\Vert -\Vert \hat v_k\Vert \big|+\frac{2}{\nu-1}\nu^{-m}.
\end{align*}
Hence, for any $\e>0$, $m$ can be chosen large enough to ensure that $\big\Vert |v_n\Vert -\Vert v_k\Vert \big|<\e$ eventually. 
\end{proof}


Denote 
\begin{align*}
\chi_{\star}=\liminf_{n\to\infty}\chi(\Theta_n)
\qquad\text{and}\qquad
\chi^{\star}=\limsup_{n\to\infty}\chi(\Theta_n). 
\end{align*}
\medskip


\section{Orbits not converging to an edge}
\label{s_notedge}

In this section we deal with the cases where we can benefit from  the sequence $(\Theta_n,\pi_n)$ 
staying long enough in the region where $\vertiii{M_p}$ is separated from one, hence leading to the geometric decay of $\Vert v_n\Vert $.

\begin{prop} 
\label{p_eq}
Almost surely on the event $\{\chi_{\star}>0\}$ the sequence $(\Theta_n,\pi_n)$ converges to an 
equilibrium point $(\Theta,q_{[\Theta]})$, where  $\Theta\in\Sigma\setminus \partial\Sigma$.  
\end{prop}

\begin{proof} By Lemma~\ref{l_norm} we have 
$\vertiii{M_{\Theta_n}}\le 1-\frac{1}{2}\chi_{\star}$ eventually for all $n$. By~\eqref{iter_v} 
we have for all $n$
\begin{align}
\Vert v_{n+1}\Vert 
&\le 
\vertiii{\Big(1-\frac{\rho}{2}\Big)M_{\Theta_n}-\frac{\rho}{2}I}\cdot \Vert v_n\Vert +\Vert U_{n+1}\Vert 
\le c\Vert v_n\Vert +\Vert U_{n+1}\Vert ,
\label{pr_v}
\end{align}
where $c=(1-\frac{\rho}{2})(1-\frac{1}{2}\chi_{\star})+\frac{\rho}{2}\in (0,1)$.
By Lemma~\ref{lemma_notgeom} and Proposition~\ref{p_error} $\Vert v_n\Vert $ exponentially decays to zero. By~\eqref{iter_incr} and Lemma~\ref{l_norm}
we also have 
\begin{align*}
\Vert \Theta_{n+1}-\Theta_n\Vert \le 2 \Vert v_n\Vert +\Vert W_{n+1}\Vert 
\end{align*} 
implying by Proposition~\ref{p_error} that $(\Theta_n)$ converges exponentially to some $\Theta$, which is is an internal point due to $\chi_{\star}>0$. 
Finally, $\pi_n\to q_{[\Theta]}$
follows from $\Vert v_n\Vert \to 0$ and $\Theta_n\to \Theta$. 
\end{proof}

\begin{prop}
\label{p_dpos_lpos}
$\P(\chi^{\star}>0, \ell>0)=0$.
\end{prop}

\begin{proof}
Let $(n_k)$ be a subsequence such that $\chi(\Theta_{n_k})>\frac{1}{2}\chi^{\star}$ for all $k$. 
By Lemma~\ref{l_norm} we have 
$$\vertiiibig{M_{\Theta_{n_k}}}< 1-\frac{1}{2}\chi^{\star}$$
for all $k$. Similarly to~\eqref{pr_v}, by~\eqref{iter_v} 
we have for all $k$
\begin{align*}
\Vert v_{n_k+1}\Vert 
\le c\Vert v_{n_k}\Vert +\Vert U_{n_k+1}\Vert ,
\end{align*}
where $c=(1-\frac{\rho}{2})(1-\frac{1}{2}\chi^{\star})+\frac{\rho}{2}\in (0,1)$. 
Letting $k\to \infty$ and using Lemma~\ref{lemma_l} and Proposition~\ref{p_error} we obtain $\ell\le c\ell$, which contradicts $c\in (0,1)$
as $\ell >0$. 
\end{proof}

\begin{prop}
\label{p_ddd}
$\P(\chi_{\star}=0, \chi^{\star}>0, \ell=0)=0$.
\end{prop}

\begin{proof} Denote 
$\text{dist}(p,\partial\Sigma)=\inf\limits_{\hat p\in \partial\Sigma} \Vert p-\hat p\Vert $ and 
observe that $\chi^{\star}>0$ implies 
\begin{align*}
\limsup_{n\to\infty}\text{\rm dist}(\Theta_n,\partial\Sigma)=d>0.
\end{align*}

Using Lemma~\ref{l_norm}, denote by $\hat c\in (0,1)$ a uniform upper bound for $\vertiii{ M_p}$
over the set $\{p: \text{\rm dist}(p,\partial\Sigma)\ge \frac{d}{4}\}$. Denote $c=(1-\frac{\rho}{2})\hat c+\frac{\rho}{2}$
and choose $\gamma$ as in Lemma~\ref{lemma_notgeom}. 
\smallskip

Let $\nu\in (1,\mu^{\frac 1 2})$. Let $m$ be such that  $\text{\rm dist}(\Theta_m,\partial\Sigma)\ge \frac{d}{2}$ and sufficiently large
so that 
\begin{align}
\label{co}
 \frac{2}{1-\gamma}\Big[\Vert v_m\Vert +\nu^{-m}\frac{\nu}{\gamma\nu-1}\Big] +\frac{\nu}{\nu-1}\nu^{-m}<\frac{d}{4},
\end{align}
which is possible as $\ell=0$, 
as well as  $\Vert U_{n+1}\Vert \le \nu^{-n}$ and $\Vert W_{n+1}\Vert \le \nu^{-n}$ for all $n\ge m$, which is possible by Proposition~\ref{p_error}. Similarly to~\eqref{pr_v}, by~\eqref{iter_v} we have 
for all $n>m$
\begin{align*}
\Vert v_{n+1}\Vert 
&\le c_n\Vert v_n\Vert +\Vert U_{n+1}\Vert \le c_n\Vert v_n\Vert +\nu^{-n},
\end{align*}
where $c_n=c$ if $\text{\rm dist}(\Theta_n, \partial\Sigma)\ge \frac{d}{4}$ and $c_n=1$ otherwise. 
\smallskip

Let $k=\inf\{n>m: \text{\rm dist}(\Theta_n, \partial\Sigma)< \frac{d}{4}\}$. Observe that $k<\infty$ due to $\chi_{\star}=0$. Then 
we have using~\eqref{iter_incr} and Lemmas~\ref{lemma_notgeom} and~\ref{l_norm} 
\begin{align}
\Vert \Theta_{k}-\Theta_m\Vert 
&\le \sum_{n=m}^{k-1}\Vert \Theta_{n+1}-\Theta_{n}\Vert 
\le 2\sum_{n=m}^{k-1}\Vert v_n\Vert +\sum_{n=m}^{k-1}\Vert W_{n+1}\Vert \notag\\
&\le 2\Big[\Vert v_m\Vert +\nu^{-m}\frac{\nu}{\gamma\nu-1}\Big]\sum_{n=m}^{k-1}\gamma^{n-m} +\sum_{n=m}^{k-1}\nu^{-n}\notag\\
&\le \frac{2}{1-\gamma}\Big[\Vert v_m\Vert +\nu^{-m}\frac{\nu}{\gamma\nu-1}\Big] +\frac{\nu}{\nu-1}\nu^{-m}<\frac{d}{4}
\label{xx4}
\end{align}
due to~\eqref{co}. Together with $\text{\rm dist}(\Theta_m,\partial\Sigma)\ge \frac{d}{2}$ this implies 
$\text{\rm dist}(\Theta_k,\partial\Sigma)\ge \frac{d}{4}$ leading to a contradiction. 
\end{proof}
\bigskip


\section{Boundary orbits}
\label{s_boundary}

We begin this section by studying the orbits of the deterministic dynamical system~\eqref{dynsys0} starting (and therefore staying) on the boundary
of the simplex. It is not hard to establish their limit behaviour which turns out to be an edge equilibrium or an edge equilibrium two-cycle.  
Using a continuity argument, we are then able to show that  if the sequence
$(\Theta_n,\pi_n)$ approaches the boundary then it has an edge equilibrium or an edge equilibrium two-cycle as a limit point.   

 
 \begin{lemma} The dynamical system~\eqref{dynsys0} is well-defined for any initial condition. 
 \end{lemma}
 
 \begin{proof} For any $p\in \Sigma\setminus V$ and $q\in\Sigma$ we have $M_pq\in\Sigma$ as its coordinates remain non-negative and 
 add up to one since so do the columns of $M_p$. 
 Further, the coordinates of $(M_p+I)q$ do not exceed $1$ since the entries of $M_p+I$ are bounded by $1$. This implies 
 $\one-q-M_p q\in \Sigma$. Finally, $(1-\rho)p+\rho(\one-q-M_p q)\in \Sigma\setminus V$ by convexity.   
 \end{proof}

For each $i$, denote by $E_i$ the edge of $\Sigma$ corresponding to the $i$-th coordinate being equal to zero. 
   
 \begin{lemma} 
 \label{l_boundary}
 Let $p_0\in E_i\setminus V$ for some $i$ and $q_0=(a_0,b_0,c_0)\in \Sigma$. 
 Then $p_n\in E_i\setminus V$ for all $n$, and $(p_n)$ converges to some $p_*\in E_i\setminus V$. 
 Furthermore, 
 \begin{align}
 \label{qqbbb}
 q_{2n}\to q_{[p_*]}-(2c_0-1)e_{-1}(p_*)\qquad\text{and}\qquad q_{2n+1}\to q_{[p_*]}+(2c_0-1)e_{-1}(p_*).
 \end{align}
 All convergences are uniform with respect to $(p_0,q_0)$. 
 \end{lemma}
 
 \begin{proof} 
 Denote $p_n=(x_n,y_n,z_n)$ and $q_n=(a_n,b_n,c_n)$. 
 Without loss of generality assume that $i=3$, i.e.\ $z_0=0$. 
Substituting this into $M_{p_0}$
we obtain that $z_1=0$ and hence $z_n=0$ for all $n$. 
 Substituting this into $M_{p_n}$ 
 we obtain 
 \begin{align}
 \label{qq1}
 a_{n+1}=y_nc_n\qquad \text{and}\qquad c_{n+1}=1-c_n
 \end{align}
 as well as 
 \begin{align}
 \label{xxxx1}
 x_{n+1}-x_n
 &=\rho(-x_n+1-a_n-y_nc_n)\notag\\
 &=\rho ((1-c_n)(1-x_n)-y_{n-1}c_{n-1})
 =\rho (1-c_n) (x_{n-1}-x_n),
 \end{align}
 which in particular implies $|x_{n+1}-x_n|\le |x_{n}-x_{n-1}|$.
 Iterating again we obtain 
 \begin{align*}
 |x_{n+1}-x_n|
 =\rho^2 (1-c_n)c_n |x_{n-1}-x_{n-2}|\le \frac 1 4 |x_{n-1}-x_{n-2}|
 \end{align*}
 since $t(1-t)\le \frac 1 4$ on $[0,1]$. This implies uniform convergence of $x_n$ to some $x_*$. 
 It follows from the sign alternation in~\eqref{xxxx1} and the decreasing size of the increments of $(x_n)$ that $x_*\in (0,1)$. Obviously, $y_n\to y_*=1-x_*$ uniformly and so 
 $p_n\to p_*=(x_*, y_*, 0)\in E_i\setminus V$ uniformly. 
 \smallskip
 
 It follows from~\eqref{qq1} that $c_{2n}=c_0$ and $c_{2n+1}=1-c_0$. Again by~\eqref{qq1} this implies that 
$a_{2n}\to y_*(1-c_0)$ and $a_{2n+1}\to y_*c_0$ uniformly. This immediately implies 
$b_{2n}\to x_*(1-c_0)$ and $b_{2n+1}\to x_*c_0$ uniformly, 
which is equivalent to~\eqref{qqbbb} by definition of $e_{-1}(p)$.
 \end{proof}
 
  It follows from Propositions~\ref{p_split} and~\ref{p_error} that 
 \begin{align}
 \label{iter_phi}
 (\Theta_{n+1},\pi_{n+1})=\Phi(\Theta_n,\pi_n)+V_{n+1},
 \end{align} 
 where $V_{n+1}=(S_{n+1},R_{n+1})$ exponentially decays to zero.
 \smallskip
 
 \begin{lemma} 
 \label{l_versub}
 If $(n_k)$ is a subsequence of indices such that $\Theta_{n_k}\to v\in V$ then $\Theta_{n_k+1}\to v$.   
 \end{lemma}

 \begin{proof} We have $\Theta_{n_k}^{\ssup i}\to 1$ for some $i$. By~\eqref{iter_lim_th} 
 \begin{align*}
 \Theta_{n_k}^{\ssup i}=(1-\rho)\Theta_{n_k-1}^{\ssup i}+\rho(1-\pi_{n_k-1}^{\ssup i}-\pi_{n_k}^{\ssup i})+S^{\ssup i}_{n_k},
\end{align*}
which implies $\Theta_{n_k-1}^{\ssup i}\to 1$ and $1-\pi_{n_k-1}^{\ssup i}-\pi_{n_k}^{\ssup i}\to 1$ by convexity as both terms are bounded above by one, and due to Proposition~\ref{p_error}. The second convergence implies $\pi_{n_k}^{\ssup i}\to 0$ due to non-negativity of all $\pi_n$.  
Taking into account Proposition~\ref{p_error} again, we obtain $\pi^{\ssup i}_{n_k+1}\to 0$ according to~\eqref{iter_pi3} and then $\Theta_{n_k+1}^{\ssup i}\to 1$ 
according to~\eqref{iter_lim_th}.  
 \end{proof}

\begin{lemma} 
\label{limit_points}
Almost surely on the event $\{\chi^{\star}=0\}\cap\mathcal{D}^c$ the sequence $(\Theta_n,\pi_n)$ 
has a limit point of the form $(p,q_{[p]}+ \ell e_{-1}(p))$ or $(p,q_{[p]}- \ell e_{-1}(p))$, where $p\in\partial\Sigma\setminus V$.
\end{lemma}

\begin{proof}
First, observe that any limit point of $(\Theta_n)$ belongs to $\partial\Sigma$ on the event $\{\chi^{\star}=0\}$. 
\smallskip

Second, on the event $\mathcal{D}^c$ the sequence $(\Theta_n)$ has a limit point that is not a vertex. Indeed, otherwise 
$(\Theta_n)$ could be split into two or three 
subsequences converging to different vertices. That, however, would contradict Lemma~\ref{l_versub}, which
makes switching between such subsequences impossible. 
\smallskip

Now we know that on the event $\{\chi^{\star}=0\}\cap\mathcal{D}^c$ we have a convergent subsequence $(\Theta_{n_k},\pi_{n_k})\to (p^{\ssup 0},q^{\ssup 0})\in (\partial\Sigma\setminus V)\times \Sigma$.
Consider the subsequence $(\Theta_{n_k-1},\pi_{n_k-1})$ and observe that it contains a convergent subsequence 
$(\Theta_{n_{k_i}-1},\pi_{n_{k_i}-1})\to (p^{\ssup 1}, q^{\ssup 1})\in \partial\Sigma\times \Sigma$ as $\chi^{\star}=0$. By Lemma~\ref{l_versub}
$p^{\ssup 1}\not\in V$ as otherwise it would imply $p^{\ssup 0}\in V$. Continuity of $\Psi$ and~\eqref{iter_phi} imply 
\begin{align*}
\Phi(p^{\ssup 1}, q^{\ssup 1})=\lim_{i\to\infty}\Phi(\Theta_{n_{k_i}-1},\pi_{n_{k_i}-1})
=\lim_{i\to\infty}(\Theta_{n_{k_i}},\pi_{n_{k_i}})=(p^{\ssup 0},q^{\ssup 0}).
\end{align*}
Iterating this procedure we obtain, for each $m$, $(p^{\ssup 0}, q^{\ssup 0})=\Phi^m(p^{\ssup m}, q^{\ssup m})$ for some 
$p^{\ssup m}\in \partial \Sigma\setminus V$ and $q^{\ssup m}\in \Sigma$. By Lemma~\ref{l_boundary} this implies that 
$(p^{\ssup 0}, q^{\ssup 0})$ is arbitrarily close to points of the form $(p,q_{[p]}+\beta e_{-1}(p))$ for some $p\in\partial\Sigma\setminus V$
and $\beta\in\R$ and hence is itself of such form.  This implies $q^{\ssup 0}=q_{[p^{\ssup 0}]}+\beta e_{-1}(p^{\ssup 0})$, and $|\beta|=\ell$
follows from Lemma~\ref{lemma_l}.
\end{proof}

\begin{lemma} 
\label{l_bdd_1}
The limit
$\ell\in [0,1]$ almost surely on the event $\{\chi^{\star}=0\}\cap\mathcal{D}^c$. 
\end{lemma}

\begin{proof} By Lemma~\ref{limit_points}, let one of the points $(p,q_{[p]}\pm \ell e_{-1}(p))$, $p\in \partial\Sigma\setminus V$, be a limit point of $(\Theta_n,\pi_n)$. Without loss of generality assume that $p=(x,y,0)$. 
Observe that $q_{[p]}\pm \ell e_{-1}(p)\in\Sigma$, and the largest value $\ell$ satisfying this is $\ell=1$, as the point 
$q_{[p]}+ e_{-1}(p)=(y,x,0)\in \partial \Sigma$ and $q_{[p]}- e_{-1}(p)=(0,0,1)\in \partial \Sigma$. 
\end{proof}

\bigskip

\section{Orbits converging to an edge}
\label{s_edge}

In this section we first find an asymptotics of $\Phi$ in small neighbourhoods of the boundary near an equilibrium point or 
an equilibrium two-cycle. 
Then, assuming that $(\Theta_n,\pi_n)$ ever comes close to such states, we prove that it never leaves, which then makes it possible to use 
the asymptotics to prove the convergence.
By symmetry it suffices to consider the edge corresponding to $z=0$. 
\smallskip


Let $(K_{\e})$ be a family of domains, indexed by a small positive parameter $\e$, such that $K_{\e_1}\subset K_{\e_2}$ whenever 
$\e_1<\e_2$. For real-valued functions $f$ and $g_1,\dots,g_k$ we write $f=O(g_1,\dots,g_k)$ 
if there is a sufficiently small $\e_0$ such that  
the ratio $\frac{f}{|g_1|+\cdots+|g_k|}$ is well-defined and bounded on 
$K_{\e_0}$ by some constant $M$. We call $(M,\e_0)$ the bounding parameters. We also use this notation for vector-valued functions if they satisfy the definition coordinatewise.  
\smallskip

Let $\delta\in (0,\frac 1 4)$ and $\theta\in [0,1]$ be fixed. 
For any $\e\in (0,\frac 1 3)$, 
denote 
\begin{align*}
K_{\delta,\e}&=\big\{(x,y,z)\in\Sigma: x\in [\delta, 1-\delta], z\in (0,\e]\big\},\\
K^{\star}_{\theta, \delta,\e}&=\Big\{(x,z,\alpha,\beta, s, r)\in \R^4\times\Sigma^2: x\in [\delta, 1-\delta], z\in (0,\e], |\alpha|\le \e, \Vert \beta|-\theta|\le \e, \Vert s\Vert +\Vert r\Vert \le \e\Big\}. 
\end{align*}

\begin{proof}[Proof of Lemma~\ref{la0}]
A simple calculation shows that $q_{[p]}$ is an eigenvector with the eigenvalue $1$. The plane $\mathcal{L}$ is invariant as 
the columns of $M_p$ add up to one. Computing the trace and the determinant of $M_p$ we observe that 
the remaining two eigenvalues satisfy 
\begin{align}
\label{eigenvalues}
\lambda_0(p)+\lambda_{-1}(p)=-1 \qquad\text{and}\qquad \lambda_0(p)\lambda_{-1}(p)=\frac{2xyz}{(x+y)(x+z)(y+z)}. 
\end{align}
It can be easily shown that the expression on the right hand side of~\eqref{eigenvalues} is bounded by $\frac 1 4$ and nonnegative on the simplex.
This implies the existence of two real eigenvalues both belonging to $[-1,0]$ and being equal to $-1$ and $0$ if and only if   
the expression on the right hand side of~\eqref{eigenvalues} equals zero.
\end{proof}

\begin{lemma} 
\label{la}
On $K_{\delta, \e}$ the following asymptotics hold:
\begin{align*}
\lambda_{0}(p)=-2z+O(z^2) 
\qquad\text{and}\qquad
\lambda_{-1}(p)=-1+2z+O(z^2)
\end{align*}
and 
\begin{align*}
e_{0}(p)=\frac 1 2(1, -1, 0)+O(z)
\qquad\text{and}\qquad
e_{-1}(p)=\frac 1 2(1-x,x,-1)+O(z).\phantom{iiii}
\end{align*}
\end{lemma}

\begin{proof}
The required asymptotics for the eigenvalues easily follows from the Taylor expansion of~\eqref{eigenvalues} with respect to $z$ around zero.  
For $p=(x,y,0)\in \partial\Sigma\setminus V$, the normalised eigenvector $e_{-1}(p)$ is given by~\eqref{e1}. Similarly,  a direct computation 
shows that $e_{0}(p)=\frac{1}{2}(1, -1, 0)$.  The asymptotics now follows from the smoothness of the eigenvectors on $K_{\delta, \e}$.
 \end{proof}
 
In what follows, we will replace $q$ by its coordinates with respect to the basis of eigenvectors $e_{-1}(p)$ and $e_0(p)$ centered at $q_{[p]}$, 
and see how those coordinates evolve under the dynamics. 
\smallskip

Let $p\in K_{\delta,\e}$. By Lemma~\ref{la0} each vector $q\in \Sigma$ has a unique representation
 \begin{align*}
 q=q_{[p]}+\alpha e_0(p)+\beta e_{-1}(p)
 \end{align*}
 with some $(\alpha,\beta)\in \R^2$. 
 \smallskip
 
 Let $r, s\in\mathcal{L}$ be parameters responsible for 
 error terms added to the dynamical system, and define
 the corresponding one-step iteration
 \begin{align}
 \label{hatp}
 \hat p&= (1-\rho)p+\rho(\one-q-M_p q)+s,\\
 \hat q&=M_p q+r.
 \label{hatq}
 \end{align}
Assuming that $\hat p\in K_{\delta, \e}$, 
denote by $(\hat \alpha,\hat \beta)$ the corresponding coordinates of 
 $\hat q$ with respect to the eigenvectors at the point $\hat p$. 
 Comparing it with the representation obtained from~\eqref{hatq} and using $M_pq_{[p]}=q_{[p]}$ we obtain   
 \begin{align}
 \label{qbar_decomp}
 \hat q=q_{[\hat p]}+\hat \alpha e_0(\hat p)+\hat \beta e_{-1}(\hat p) 
 = q_{[p]}+\alpha \lambda_0(p)e_0(p)+\beta \lambda_{-1}(p)e_{-1}(p)+r.
 \end{align}
Using~\eqref{hatp} we get 
\begin{align}
\label{qqqq}
q_{[p]}-q_{[\hat p]}
&=q_{[p]}-\frac{\one-\hat p}{2}
=\rho \Big[q_{[p]}-\frac{q+M_p q}{2}\Big]+\frac{s}{2}
=-\frac{\rho}{2} (I+M_p)(q-q_{[p]})+\frac{s}{2}
\notag\\
&
=-\frac{\alpha\rho}{2} (1+\lambda_0(p))e_0(p)-\frac{\beta\rho}{2} (\lambda_{-1}(p)+1)e_{-1}(p)+\frac{s}{2}.
\end{align}
Combining this with~\eqref{qbar_decomp} we 
obtain a system of three linear equations  for the two variables $\hat\alpha$ and $\hat\beta$:
 \begin{align*}
 \hat \alpha e_0(\hat p)+\hat \beta e_{-1}(\hat p) =
 \alpha \Big[\Big(1-\frac{\rho}{2}\Big)\lambda_0(p)-\frac{\rho}{2} \Big]e_0(p)
 +\beta \Big[\Big(1-\frac{\rho}{2}\Big)\lambda_{-1}(p)-\frac{\rho}{2} \Big]e_{-1}(p)
 +r+\frac{s}{2}.
 \end{align*}
We can use any two equations $i\neq j$ out of the three to solve the system. Via Cramer's rule we obtain 
\begin{align}
\label{eq_ab1}
\hat\alpha =D^{-1}
\Big\{ &\alpha\Big[\Big(1-\frac{\rho}{2}\Big)\lambda_0(p)-\frac{\rho}{2} \Big]
D_{\hat\alpha\alpha}
+\beta \Big[\Big(1-\frac{\rho}{2}\Big)\lambda_{-1}(p)-\frac{\rho}{2} \Big]
D_{\hat\alpha\beta}
+O(r,s)\Big\},\\
\hat\beta =
D^{-1}\Big\{
&\alpha\Big[\Big(1-\frac{\rho}{2}\Big)\lambda_0(p)-\frac{\rho}{2} \Big]
D_{\hat\beta\alpha}
+\beta\Big[\Big(1-\frac{\rho}{2}\Big)\lambda_{-1}(p)-\frac{\rho}{2} \Big]
D_{\hat\beta\beta}
+O(r,s)\Big\},
\label{eq_ab2}
\end{align}
where 
\begin{align*}
D= \left|\begin{array}{cc}
e_0^{\ssup i}(\hat p) & e_{-1}^{\ssup i}(\hat p)\\
e_0^{\ssup j}(\hat p) & e_{-1}^{\ssup j}(\hat p)
\end{array}\right|
\end{align*}
and 
\begin{align*}
D_{\hat\alpha\alpha}
&=\left|\begin{array}{cc}
e^{\ssup i}_0(p) & e_{-1}^{\ssup i}(\hat p)\\
e^{\ssup j}_0(p) & e_{-1}^{\ssup j}(\hat p)
\end{array}\right|, 
\qquad D_{\hat\alpha\beta}
=\left|\begin{array}{cc}
e^{\ssup i}_{-1}(p) & e_{-1}^{\ssup i}(\hat p)\\
e^{\ssup j}_{-1}(p) & e_{-1}^{\ssup j}(\hat p)
\end{array}\right|, \\
D_{\hat\beta\alpha}
&=\left|\begin{array}{cc}
e_{0}^{\ssup i}(\hat p) & e^{\ssup i}_{0\phantom{-}}(p) \\
e_{0}^{\ssup j}(\hat p) & e^{\ssup j}_0(p) 
\end{array}\right|, 
\qquad D_{\hat\beta\beta}
=\left|\begin{array}{cc}
e_{0\phantom{-}}^{\ssup i}(\hat p) & e^{\ssup i}_{-1}(p)\\
e_{0}^{\ssup j}(\hat p) & e^{\ssup j}_{-1}(p) 
\end{array}\right|.
\end{align*}
Observe that the absolute values of the determinants do not depend on the choice of $i\neq j$. Indeed, each determinant is 
a coordinate (with a plus or a minus sign depending on $i,j$) of the vector product of the two eigenvectors used to build it. The choice of the coordinate depends on $i,j$ but the coordinates themselves are all equal since the vector product is orthogonal to $\mathcal{L}$ where both eigenvectors lie. The difference in sign will be cancelled by the fact that the same sign will be used in all determinants. 
\smallskip

Denote $p=(x,y,z)$ and $\hat p=(\hat x, \hat y, \hat z)$. Let $u=(x,z,\alpha,\beta,r,s)$. 
We will now find the asymptotics of the dynamical system with respect to the small parameters $z, \alpha, s, r$ and, for the case $\theta=0$, 
also $\beta$.

\begin{lemma} 
\label{newas}
The following asymptotics hold on $K^{\star}_{\theta,\delta,\e}$. 

(a) If $\theta>0$ then 
\begin{align*}
\hat x 
&=x-\frac{1}{2}\rho\alpha+\phi_1(u),
\qquad\qquad
\phantom{aaaa}\hat z =z(1+\rho\beta)+\phi_2(u),\\
\hat \alpha
&=-\frac 1 2 \rho (1-\beta)\alpha+\phi_3(u),
\qquad \qquad
\hat \beta =-\beta+\phi_4(u),
\end{align*}
where $\phi_1(u)=O(z,s)$, $\phi_2(u)=O(\alpha z,z^2,s)$, $\phi_3(u)=O(z,s,r)$ and $\phi_4(u)=O(z,s,r)$.
\smallskip

(b) If $\theta=0$ then
\begin{align*}
\hat x 
&=x-\rho(1-x)\beta z+\phi_1(u),
\qquad\qquad
\hat z 
=z(1+\rho\beta)+\phi_2(u),\\
\hat \alpha
&=-\frac 1 2 \rho \alpha+\phi_3(u),
\qquad\qquad
\phantom{aaaaaaa}\hat \beta 
=-\beta(1-(2-\rho)z)+\phi_4(u),
\end{align*}
where $\phi_1(u)=O(\alpha,\beta z^2,s)$, $\phi_2(u)=O(\alpha z,\beta z^2,s)$,  $\phi_3(u)=O(\alpha^2, \alpha z, \alpha\beta, z\beta^2,s,r)$ and\\
 $\phi_4(u)=O(\alpha^2, \alpha\beta, \beta^2 z,\beta z^2,s,r)$.
\end{lemma}

\begin{proof} Using~\eqref{qqqq} together with the definition of $q_{[p]}$ and $q_{[\hat p]}$ we get  
\begin{align*}
\hat p =p-\rho\alpha (1+\lambda_0(p)) e_0(p)-\rho\beta (1+\lambda_{-1}(p)) e_{-1}(p)+s.
\end{align*}
This implies the required asymptotics for $\hat x$ and $\hat z$ by Lemma~\ref{la} for both cases (a) and (b). 
For the asymptotics of $\hat\alpha$ and $\hat\beta$ we consider the cases separately. 
\smallskip
 
 (a) Denote $e_0=\frac 1 2 (1,-1,0)$ and $e_{-1}(x)=\frac 1 2 (1-x,x,-1)$. Using the asymptotics for $\hat x$ and $\hat z$   together with Lemma~\ref{la}
we have 
\begin{align}
\label{e01}
e_0(\hat p)&=e_0+O(\hat z)=e_0+O(z,s)\\
e_{-1}(\hat p)&=e_{-1}(\hat x)+O(\hat z)=e_{-1}(x)-(\hat x-x)e_0+O(z,s)=e_{-1}(x)+\frac{\rho\alpha}{2}e_0+O(z,s).
\label{e11}
\end{align}
Fix some $i\neq j$ and denote
\begin{align*}
\Delta=\left|\begin{array}{cc}
e_0^{\ssup i} & e_{-1}^{\ssup i}(x)\\
e_0^{\ssup j} & e_{-1}^{\ssup j}(x)
\end{array}\right|=\pm \frac 1 4,
\end{align*}
where the sign depends on the choice of $i,j$. 
Substituting~\eqref{e01} and~\eqref{e11}  we obtain
\begin{align*}
D=\Delta+O(z,s)
\end{align*}
as the term $\frac{1}{2}\rho\alpha$ from~\eqref{e11} vanishes due to the corresponding determinant being zero. 
Similarly, we use ~\eqref{e01} and~\eqref{e11}  as well as Lemma~\ref{la} for $e_0(p)$ and $e_{-1}(p)$ to compute 
\begin{align*}
D_{\hat\alpha\alpha}&=\Delta+O(z,s),
&D_{\hat\alpha\beta}
&=-\frac{\rho\alpha}{2}\Delta
+O(z, s),\\
D_{\hat\beta\alpha}
&=O(z, s),
& D_{\hat\beta\beta}
&=\Delta+O(z, s).
\end{align*}
Further, by Lemma~\ref{la} we have
\begin{align*}
\Big(1-\frac{\rho}{2}\Big)\lambda_0(p)-\frac{\rho}{2}=-\frac{\rho}{2}+O(z) 
\quad\text{and}\quad 
\Big(1-\frac{\rho}{2}\Big)\lambda_{-1}(p)-\frac{\rho}{2}=-1+O(z).
\end{align*}
Substituting everything into ~\eqref{eq_ab1} and~\eqref{eq_ab2} 
we obtain the required formulas. 
\smallskip

(b) Observe that $\hat p=p+O(\alpha,\beta z, s)$ due to the asymptotics for $\hat x$ and $\hat z$. 
Hence the ratios of the determinants in~\eqref{eq_ab1} and~\eqref{eq_ab2} are equal to either one 
or zero 
up to the error term $O(\alpha,\beta z, s)$. Combining this with Lemma~\ref{la}  providing the 
asymptotics for $\lambda_0(p)$ and
\begin{align*}
\Big(1-\frac{\rho}{2}\Big)\lambda_{-1}(p)-\frac{\rho}{2}
=-1+z(2-\rho)+O(z^2).
\end{align*}
we obtain the required formulas for $\hat\alpha$ and $\hat\beta$. 
\end{proof}

Now we are ready to apply the asymptotic analysis to the sequence $(\Theta_n, \pi_n)$. Denote $\Theta_n=(x_n,y_n,z_n)$ and let $(\alpha_n,\beta_n)$ be the coordinates of $v_n\in \mathcal{L}$ with respect to the eigenvectors $e_0(\Theta_n)$, 
$e_{-1}(\Theta_n)$ whenever 
those are well-defined. Denote $r_n=R_{n+1}$ and 
$s_n=S_{n+1}$, which is consistent with our assumptions as~\eqref{iter_lim_pi} and ~\eqref{iter_lim_th} entail that $R_{n+1},S_{n+1}\in\mathcal{L}$. Finally, denote $u_n=(x_n,z_n,\alpha_n,\beta_n,s_n,r_n)$.
\smallskip

Recall that $\delta\in (0,\frac 1 4)$ and fix $\nu\in (1,\mu^{\frac 1 2})$. 
Let $(M,\e_0)$ be the bounding parameters of all $\phi_i$ in Lemma~\ref{newas}, where we assume $\e_0\in (0,\frac 1 3)$ and 
$M\ge 100$. 
\smallskip

Let $\gamma\in (\nu^{-\frac 1 2},1)$ be sufficiently close to one, $\eta>0$ be sufficiently small, and $\e\in (0,\e_0)$ be sufficiently small, where each of these 
parameters may depend on all previously introduced ones. Those assumptions will be extensively used in the sequel.
\smallskip

Let $m$ be such that
\begin{align}
\label{inside0}
x_m\in [2\delta, 1-2\delta],\quad z_m<\e^2,\quad |\alpha_m|<\e^2,\quad |\beta_m\pm \theta|<\e^2
\end{align} 
where $\pm$ is understood to be $+$ if the sign of $\beta_m$ is negative and $-$ otherwise, and
\begin{align}
\label{inside_error}
\Vert r_n\Vert +\Vert s_n\Vert <\nu^{-n}<\e^4
\qquad\text{for all }n\ge m.
\end{align}
Observe that we do not claim at this stage that such $m$ exists, but it will easily follow from~Lemma~\ref{limit_points} and Proposition~\ref{p_error}
when we need it later on in Propositions~\ref{p_lpos} and~\ref{p_pos2}. 
\smallskip

It is easy to see that due to $\gamma\nu>1$~\eqref{inside_error} implies, 
for all $n\ge m$, that 
\begin{align}
\label{rs}
\Vert r_n\Vert +\Vert s_n\Vert <\e^4\gamma^{n-m}.
\end{align}

Let 
\begin{align*}
\tau=\inf\{n> m: u_n\not\in K^{\star}_{\theta,\delta,\e}\}.
\end{align*}

In the following two lemmas (serving the edge equilibrium two-cycle case and the edge equilibrium case, respectively) we show that if the dynamical system gets well inside~$K^{\star}_{\theta, \delta,\e}$ then it will stay there forever. 
Moreover, we estimate the decay rate of the small parameters and of the increments of $(x_n)$. 

\begin{lemma} 
\label{pall}
Suppose $\theta>0$ and assume that
$m$ satisfying~\eqref{inside0} and~\eqref{inside_error} exists. Then $\tau=\infty$ and for all $n\ge m$
\begin{align}
z_n & < \eta\e\gamma^{n-m} \label{pz}\\
|\beta_n\pm \theta(-1)^{n-m}| & <\e \label{pb}\\
|\alpha_n|&< \sqrt{\eta} \e \gamma^{n-m} \label{pa}\\
|x_{n+1}-x_n|&<   M\e\gamma^{n-m}\label{px}
\end{align} 
\end{lemma}

\begin{proof} 
We will first prove the estimates for  all $m\le n\le \tau$ and then show that $\tau=\infty$. 
\smallskip

(a) 
We will prove~\eqref{pz} by induction. It is clearly true for $n=m$\gnote{as long as $\e<\eta$}\nnote{ok}, and it is true for $n=m+1$  since\gnote{as long as $\e<\frac{1}{3M}$ and $\e<\frac{\gamma\eta}{4}$} \nnote{ok}
due to Lemma~\ref{newas} and $u_m\in K^{\star}_{\theta,\delta,\e}$ we have 
\begin{align*}
z_{m+1}
&\le z_m(1+\rho|\beta_m|)+M\big(|\alpha_m|z_m+z_m^2+\Vert s_m\Vert \big)
\le 3\e^2+3M\e^4<4\e^2<\gamma\eta\e.
\end{align*}
Suppose~\eqref{pz}  is true for some $m\le n\le \tau-2$ and let us show that it is true for $n+2$. 
\smallskip

For $i\in\{n,n+1\}$ using Lemma~\ref{newas}, $u_i\in K^{\star}_{\theta,\delta,\e}$ and~\eqref{rs} we have 
\begin{align}
\label{kokoz}
z_{i+1}
\le z_i(1+\rho\beta_i)+M\big(|\alpha_i|z_i+z_i^2+\Vert s_i\Vert \big)
\le z_i(1+\rho\beta_i+2M\e)+M\e^4\gamma^{i-m}
\end{align}
and 
\begin{align}
\label{kokob}
\beta_{i+1}
&\le -\beta_i+M(z_i+\Vert s_i\Vert +\Vert r_i\Vert )\le -\beta_i+2M\e. 
\end{align}
Using~\eqref{kokoz} twice we obtain
\begin{align*}
z_{n+2}
&\le  z_n(1+\rho\beta_n+2M\e)(1+\rho\beta_{n+1}+2M\e)+5M\e^4\gamma^{n-m}. 
\end{align*}
Observe that both terms in the brackets are positive. 
Hence we can substitute~\eqref{kokob} and get
\begin{align*}
z_{n+2}
&\le  z_n(1+\rho\beta_n+2M\e)(1-\rho\beta_n+4M\e)+5M\e^4\gamma^{n-m}\\
&\le  z_n(1\pm \rho \theta +3M\e)(1\mp \rho \theta +5M\e)+5M\e^4\gamma^{n-m}\\
&\le \eta\e\gamma^{n-m}\big[1-\rho^2 \theta^2 +20 M\e\big]
\le \eta\e\gamma^{n+2-m}.
\end{align*}
\gnote{as long as $\gamma>\sqrt{1-\gamma^2}$ and $\e<\frac{\rho^2\theta^2-(1-\gamma^2)}{20M}$.}\nnote{ok}
\smallskip

(b) We prove~\eqref{pb} and~~\eqref{pa} by induction. Both statements are clearly true for $n=m$\gnote{as long as $\e<\sqrt{\eta}$}\nnote{ok}. Suppose they are true for some $m\le n< \tau$ 
and let us show that they are true for $n+1$. 
Using Lemma~\ref{newas}, \eqref{rs}, \eqref{pz},  and $u_n\in K^{\star}_{\theta, \delta,\e}$ we obtain\gnote{as long as $\eta<\frac{1-\gamma}{3M}$}  \nnote{ok}
\begin{align*}
|\beta_{n+1}\pm \theta(-1)^{n+1-m}|
&\le |\beta_{n}\pm \theta(-1)^{n-m}|+M(z_n+\Vert s_n\Vert +\Vert r_n\Vert )\\
&\le |\beta_{n}\pm \theta(-1)^{n-m}|+M\eta\e\gamma^{n-m}+M\e^4\gamma^{n-m}\\
&\le |\beta_{m}\pm \theta|+2M\eta\e\sum_{i=m}^{n}\gamma^{i-m}
\le \e^2+\frac{2M\eta\e}{1-\gamma}
<\e.
\end{align*}
and\gnote{as long as $\e<\rho<\gamma$ and $\eta<\min\{\frac{\gamma}{M},\left[\frac{2(\gamma-\rho)}{1+3M}\right]^2\}$} 
\nnote{ok}
\begin{align*}
|\alpha_{n+1}|
&\le \frac 1 2  |\alpha_n|\rho |1-\beta_n|+M(z_n+\Vert s_n\Vert +\Vert r_n\Vert )\\
&\le \frac 1 2 \sqrt{\eta} \e \gamma^{n-m}\rho(1+\theta+\e)+M\eta \e\gamma^{n-m}+M\e^4\gamma^{n-m}\\
&=\frac 1 2 \sqrt{\eta} \e \gamma^{n-m}\big[\rho(1+\theta+\e)+3M\sqrt{\eta}\big]
\le \sqrt{\eta} \e \gamma^{n+1-m}.
\end{align*}

(c) Using Lemma~\ref{newas}, \eqref{pz}, \eqref{pa} and $u_n\in K^{\star}_{\theta, \delta,\e}$ we have for all $m\le n< \tau$
\gnote{as long as $\eta<\frac{1}{3}$}\nnote{ok}
\begin{align*}
|x_{n+1}-x_n|
&\le M(|\alpha_n|+z_n+\Vert s_n\Vert )
\le M\e \gamma^{n-m}( \sqrt{\eta} + \eta  +\e^3)
\le M\e \gamma^{n-m}.
\end{align*}

(d) Suppose $\tau<\infty$. Observe that $z_{\tau}$, $\beta_{\tau}$, $\alpha_{\tau}$ satisfy the conditions to belong to 
$K_{\theta,\delta,\e}^{\star}$ by~\eqref{pz}, \eqref{pb} and~\eqref{pa} with $n=\tau$, respectively,  and so do $s_{\tau}$, $r_{\tau}$. It remains to show that $x_{\tau}$ 
satisfies the condition as well, thus contradicting the definition of $\tau$. 
By~\eqref{px} we have
\gnote{as long as $\e<\frac{\delta(1-\gamma)}{M}$} 
\nnote{ok}
\begin{align*}
|x_{\tau}-x_m|\le \sum_{n=m}^{\tau-1}|x_{n+1}-x_n|\le M\e\sum_{n=m}^{\tau-1}\gamma^{n-m}
\le \frac{M\e}{1-\gamma}<\delta.
\end{align*}
Since $x_m\in [2\delta,1-2\delta]$ this implies that $x_{\tau}\in [\delta,1-\delta]$.
\end{proof}

\begin{lemma} 
\label{double}
Suppose $\theta=0$ and assume that
$m$ satisfying~\eqref{inside0} and~\eqref{inside_error} exists. Then $\tau=\infty$ and 
\begin{align}
\label{maxi}
|\alpha_n|&<\max\big\{ z_n |\beta_n|, \e^2\gamma^{n-m}\big\}\\
|x_{n+1}-x_n|&<3M \max\big\{ z_n |\beta_n|, \e^2\gamma^{n-m}\big\}
\label{0xx}
\end{align}
for all $n\ge m$. Furthermore, 
\begin{align}
\label{serser}
\sum_{n=m}^{\infty}z_n|\beta_n|<\infty.
\end{align}
\end{lemma}

\begin{proof} 
(a) Let us prove~\eqref{maxi} by induction for all $m\le n\le \tau$.
It is clearly true for $n=m$. 
Suppose it is true for some $m\le n<\tau$. 
Using Lemma~\ref{newas}, $u_n\in K^{\star}_{0,\delta,\e}$ and~\eqref{rs} 
we proceed as follows. 
\smallskip

Observe that
\gnote{as long as $\e\le\frac{\rho}{10M}$,} 
\nnote{ok}
\begin{align}
\label{maxim}
|\alpha_{n+1}|
&\le |\alpha_n|\frac{\rho}{2}+M(\alpha_n^2+|\alpha_n| z_n+|\alpha_n\beta_n|+z_n\beta_n^2+\Vert s_n\Vert +\Vert r_n\Vert )\notag\\
&\le |\alpha_n|\Big(\frac{\rho}{2}+3M\e\Big)+M\e z_n|\beta_n|+M\e^4\gamma^{n-m}
\le \rho\max\big\{z_n |\beta_n|, \e^2 \gamma^{n-m} \big\}.
\end{align}

If $\e^2\gamma^{n-m}$ is the largest term under the maximum  then we get
\gnote{as long as $\gamma>\rho$}
\nnote{ok} 
\begin{align*}
|\alpha_{n+1}|\le \e^2\gamma^{n+1-m}.
\end{align*}

Now consider the case when $z_n|\beta_n|$  is the largest term under the maximum. 
We have
\gnote{as long as $\e<\frac{1-\sqrt{\rho}}{4M}$} 
\nnote{ok}
\begin{align}
z_{n+1} 
&\ge z_n(1-\rho |\beta_n|)-M(|\alpha_n| z_n+|\beta_n| z_n^2+\Vert s_n\Vert )
\ge z_n(1-\e)-M(2\e z_n+\e^2 \gamma^{n-m})\notag\\
&\ge z_n(1-\e)-M(2\e z_n+z_n |\beta_n|)
\ge z_n(1-4M\e)>\sqrt \rho z_n
\label{zzzz}
\end{align}
and
\gnote{as long as $\e<\frac{1-\sqrt{\rho}}{7M}$}
\nnote{ok}
\begin{align}
|\beta_{n+1}|
&\ge |\beta_n|(1-2z_n)-M(\alpha_n^2+|\alpha_n\beta_n|+z_n\beta_n^2+|\beta_n|z_n^2+\Vert s_n\Vert +\Vert r_n\Vert )\notag\\
&\ge |\beta_n|(1-2\e)-M\big(2\e|\alpha_n|+2\e|\beta_n|+\e^2 \gamma^{n-m}\big)\notag\\
&\ge |\beta_n|(1-2\e)-M(2\e z_n |\beta_n|+2\e|\beta_n|+\e^2 z_n|\beta_n|)
\ge |\beta_n|(1-7M\e)>\sqrt \rho |\beta_n|.
\label{bbbb}
\end{align}
Combining~\eqref{maxim}, ~\eqref{zzzz} and~\eqref{bbbb} we obtain 
\begin{align*}
|\alpha_{n+1}|\le z_{n+1}|\beta_{n+1}|
\end{align*}
completing the induction. 
\smallskip

(b) Let us prove~\eqref{0xx} by induction for all $m\le n< \tau$. For any $m\le n<\tau$ let us distinguish between the two cases according to where the maximum in~\eqref{maxi} is achieved. 
Using part (a), Lemma~\ref{newas} and $u_n\in K^{\star}_{0,\delta,\e}$ we obtain the following. 
\smallskip

If the second term is the largest then
\gnote{as long as $\e<\frac{1}{5M}$}
\nnote{ok}
\begin{align}
\label{3phin}
|\phi_1(u_n)|&\le 2M\e^2\gamma^{n-m}
\qquad\text{and}\qquad 
|\phi_i(u_n)|\le 5M\e^3\gamma^{n-m}<\e^2\gamma^{n-m}
\quad\text{ for }i\in\{2,3,4\}. 
\end{align}
In particular,
\begin{align}
\label{xinc}
|x_{n+1}-x_n|\le (1-x_n)z_n|\beta_n|+2M\e^2\gamma^{n-m}<3M\e^2\gamma^{n-m}.
\end{align}

If the first term is the largest then 
we have 
\begin{align*}
\Vert s_n\Vert +\Vert r_n\Vert <\e^4 \nu^{m-n}<(\e^2 \gamma^{n-m})^2<  z_n^2 |\beta_n|
\end{align*}
and hence
\begin{align}
|\phi_1(u_n)|
\le 2M |\beta_n|z_n, 
\qquad
|\phi_2(u_n)|\le 3M |\beta_n| z_n^2,  
\qquad
|\phi_4(u_n)|\le 3Mz_n\beta_n^2+2M|\beta_n|z_n^2. 
\label{2phin}
\end{align} 
In particular,
\begin{align}
\label{xinc2}
|x_{n+1}-x_n|\le (1-x_n)z_n|\beta_n|+2M|\beta_n|z_n<3Mz_n|\beta_n|.
\end{align}
Combining~\eqref{xinc} and~\eqref{xinc2}, we obtain~\eqref{0xx} for all $m\le n<\tau$.
\smallskip

(c)
Let  
\begin{align*}
\eta=\inf\big\{n\ge m: z_n|\beta_n|\ge \e^2\gamma^{n-m}\big\}. 
\end{align*}
and let us show that if $\tau<\infty$ then $\eta<\tau$. 
\smallskip

Denote $\tau\wedge\eta=\min\{\tau,\eta\}$ and consider $m\le n<\tau\wedge\eta$. 
Using part (a), Lemma~\ref{newas} and \eqref{3phin} we have  
\begin{align}
|\alpha_{n+1}|
&\le |\alpha_n|\frac{\rho}{2}+\e^2\gamma^{n-m}
\le 2\e^2\gamma^{n-m}<\e.
\label{na}
\end{align}
Additionally taking into account that $u_i\in K^{\star}_{0,\delta,\e}$ for all $m\le i\le n$ we also have 
\begin{align}
|z_{n+1}-z_m|
\le \sum_{i=m}^{n}|z_{i+1}-z_i| 
&\le \sum_{i=m}^{n}\big[ \rho z_i|\beta_i|+\e^2\gamma^{i-m}\big]
\le 2 \e^2 \sum_{i=m}^{n}\gamma^{i-m}
<\frac{2\e^2}{1-\gamma}
< \frac{\e}{6}
\label{nz}
\end{align}
\gnote{as long as $\e<\frac{1-\gamma}{12}$,}
\nnote{ok}
as well as
\begin{align}
|\beta_{n+1}-\beta_m|
&\le \sum_{i=m}^{n}|\beta_{i+1}+\beta_i| +2\e^2
\le  \sum_{i=m}^{n}\big[2 z_i|\beta_i|+\e^2\gamma^{i-m}\big]+2\e^2\notag\\
&\le 3\e^2 \sum_{i=m}^{n}\gamma^{i-m}+2\e^2
=\frac{5\e^2}{1-\gamma}
< \frac{\e}{3}
\label{nb}
\end{align}
\gnote{as long as $\e<\frac{1-\gamma}{15}$}
\nnote{ok}
and by~\eqref{xinc}
\gnote{as long as $\e<\frac{1-\gamma}{2M}$}
\nnote{ok}
\begin{align}
|x_{n+1}-x_m|
\le \sum_{i=m}^{n}|x_{i+1}-x_i|
&< 2M\e^2\sum_{i=m}^{n}\gamma^{i-m}<\frac{2M\e^2}{1-\gamma}< \e.
\label{nx}
\end{align}

Suppose now $\tau<\infty$. It follows from~\eqref{na}--\eqref{nx} used with $n=\tau\wedge\eta-1$ 
that  $u_{\tau\wedge\eta}$ satisfies the conditions to belong to $K_{0,\delta,\e}^{\star}$
\gnote{as long as $\e<\delta$}
\nnote{ok}. 
Hence $\tau\wedge\eta\neq \tau$ and so  $\eta<\tau$. 
\smallskip

(d) Consider the case $\eta<\infty$. 
\smallskip

First, let us show by induction that for all $\eta \le n< \tau$
\begin{align}
\label{etaeta}
z_n|\beta_n|\ge \e^2\gamma^{n-m}.
\end{align}
It is indeed true for $n=\eta$ by the definition of $\eta$. Suppose it is true for some $n$. Using the same arguments as in~\eqref{zzzz}
and~\eqref{bbbb} we have 
\begin{align*}
z_{n+1}\ge z_n(1-4M\e)
\qquad \text{and}\qquad 
|\beta_{n+1}|\ge |\beta_n|(1-6M\e)
\end{align*}
\gnote{As long as $\e<\frac{5-\sqrt{1+24\gamma}}{24M}$}
concluding the induction by 
\begin{align*}
z_{n+1}|\beta_{n+1}|\ge (1-4M\e)(1-6M \e) z_n|\beta_n|\ge \e^2\gamma^{n+1-m}.
\end{align*}

Further, for all $\eta\le n<\tau$, using Lemma~\ref{newas} and \eqref{2phin},  we 
observe that $(\beta_n)$ has an alternating sign since
\gnote{as long as $\e<\frac{1}{6M}$}
\nnote{ok}
\begin{align*}
\frac{\beta_{n+1}}{\beta_n}
\le -1+2z_n+3Mz_n|\beta_n|+2Mz_n^2
\le -1+\e(2+5M\e)<0
\end{align*}
and $|\beta_n|$ is decreasing as
\begin{align}
\label{bdecr}
|\beta_{n+1}|
\le  |\beta_n|(1-z_n)+3Mz_n\beta_n^2+2M|\beta_n|z_n^2
\le  |\beta_n|\big(1-z_n\big[1-5M\e]\big)\le |\beta_n|,
\end{align}
implying
\begin{align}
\sum_{i=\eta}^{n}z_i|\beta_i|
&=\Big|\sum_{i=\eta}^{n}z_i\beta_i(-1)^i\Big|
=\Big|\sum_{i=\eta}^{n}\big[\beta_{i+1}+\beta_i-\phi_4(u_i)\big](-1)^i\Big|\notag\\
&\le \Big|\sum_{i=\eta}^{n}\Big[\beta_{i+1}+\beta_i\Big](-1)^i\Big|+\sum_{i=\eta}^{n}\big[3Mz_i\beta_i^2+2M|\beta_i|z_i^2\big]\notag\\
&\le |\beta_{n+1}|+|\beta_{\eta}|+5M\e \sum_{i=\eta}^{n}z_i|\beta_i|
\le 2|\beta_{\eta}|+5M\e \sum_{i=\eta}^{n}z_i|\beta_i|. 
\label{partsum}
\end{align}
It follows from~\eqref{nb} with $n=\eta-1$ that $|\beta_{\eta}|<\frac{\e}{6}+\e^2<\frac{\e}{5}$,
\gnote{as long as $\e<\frac{1}{30}$} 
\nnote{ok}
which together with~\eqref{partsum} implies
\gnote{, as long as $\e<\frac{1}{15M}$} 
\nnote{ok}
\begin{align}
\label{serzb}
\sum_{i=\eta}^{n}z_i|\beta_i|<(1-5M\e)^{-1}\frac{2\e}{5}<\frac{3\e}{5}.
\end{align}

Using Lemma~\ref{newas}, \eqref{2phin} and~\eqref{serzb} we have
\gnote{, as long as $\e<\frac{1-\rho}{3M}$,} 
\nnote{ok}
\begin{align}
\label{ztau}
|z_{n+1}-z_{\eta}|
\le \sum_{i=\eta}^{n}|z_{i+1}-z_i| 
&\le \sum_{i=\eta}^{n}\Big[ \rho z_i|\beta_i|+3M|\beta_i| z_i^2\Big]
\le (\rho+3M\e)\sum_{i=\eta}^{n}z_i|\beta_i|<\frac{3\e}{5}
\end{align}
and, estimating as in~\eqref{xinc2}, 
\begin{align}
|x_{n+1}-x_{\eta}|
\le \sum_{i=\eta}^{n}|x_{i+1}-x_i|
\le  3M\sum_{i=\eta}^{n}z_i|\beta_i|<3M\e. 
\label{xtau}
\end{align}

We will now combine all those estimates, using them with 
$n=\eta-1$ and $n=\tau-1$, to show that $\tau=\infty$. Suppose $\tau<\infty$. 
By~\eqref{nx} we have $x_{\eta}\in [2\delta-\e, 1-2\delta+\e]$ and hence by~\eqref{xtau} we have $x_{\tau}\in [\delta, 1-\delta]$
\gnote{as long as $\e<\frac{\delta}{1+3M}$}
\nnote{ok}
. 
By~\eqref{bdecr} we have $|\beta_{\tau}|\le |\beta_{\eta}|<\frac{\e}{5}<\e$.  
By~\eqref{nz} and~\eqref{ztau} we have\gnote{as long as $\e<\frac{\delta}{1+3M}$} 
\nnote{ok}
\begin{align*}
z_{\tau}\le |z_{\tau}-z_{\eta}|+|z_{\eta}-z_{m}|+z_{m}<\frac{3\e}{5}+\frac{\e}{6}+\e^2<\e. 
\end{align*}
By part (a) with $n=\tau$, we have $|\alpha_{\tau}|\le z_{\tau}|\beta_{\tau}|<\e^2<\e$. This implies that $u_{\tau}\in K_{0,\delta,\e}^{\star}$ leading to a contradiction. 
\smallskip

Finally, it follows from $\tau=\infty$ and~\eqref{serzb} that the series in~\eqref{serser} converges.
\smallskip

(e) Consider the case $\eta=\infty$. 
\smallskip

Observe that $\tau=\infty$ by (c) and the series in~\eqref{serser} converges as it is dominated by a geometric series. 
\end{proof}

Now we can combine the previous two lemmas with Lemma~\ref{limit_points}, which ensures that $(\Theta_n, \pi_n)$
does indeed get close to an edge equilibrium or an edge equilibrium two-cycle. 

\begin{prop} 
\label{p_lpos}
Almost surely on the event
$\{\chi^{\star}=0, \ell=0\}\cap \mathcal{D}^c$ the sequence $(\Theta_n)$ converges 
to some $p\in \partial\Sigma\setminus V$ and the sequence $(\pi_n)$ converges
to $q_{[p]}$.
\smallskip

Almost surely on the event $\{\chi^{\star}=0,\ell>0\}\cap \mathcal{D}^c$ the sequence $(\Theta_n)$ converges to some $p\in \partial\Sigma\setminus V$, and the sequence $(\pi_n)$ approaches the two-cycle 
$\{q_{[p]}\pm \ell e_{-1}(p)\}$.  
\end{prop}

\begin{proof} 
Consider the event $\{\chi^{\star}=0\}\cap \mathcal{D}^c$ and let $\theta=\ell$. By Lemma~\ref{limit_points} $(\Theta_n, \pi_n)$ has a limit point of the form $(p, q_{[p]}\pm \theta e_{-1}(p))$, where 
$p\in\partial\Sigma\setminus V$. Without loss of generality we may assume that $p=(x,1-x,0)$, $x\in (0,1)$. Let $\delta$
be sufficiently small so that $x\in (2\delta, 1-2\delta)$, and let the parameters $\nu, \gamma,\eta,\e$ be chosen accordingly as above. Then, and due to Proposition~\ref{p_error}, 
$m$ satisfying~\eqref{inside0} and~\eqref{inside_error} exists. Now, using either Lemma~\ref{pall} or Lemma~\ref{double} depending on the value of $\theta$, 
we obtain that $\alpha_n\to 0$ and that the series of increments of $(x_n)$
converges implying convergence of $(x_n)$ to $x$ since $x$ was already a limit point. Further, $z_n\to 0$ due to $\chi^{\star}=0$ and $\tau=\infty$. 
\smallskip

On the event $\{\chi^{\star}=0, \ell=0\}\cap \mathcal{D}^c$ we have $\beta_n\to 0$ as $\Vert v_n\Vert \to 0$ by Lemma \ref{lemma_l}, implying $\pi_n\to q_{[p]}$. 
\smallskip

On the event $\{\chi^{\star}=0, \ell>0\}\cap \mathcal{D}^c$ we obtain that 
$(\beta_n)$ approaches the two-cycle $\{\pm\ell\}$ since $\Vert v_n\Vert \to \ell$ by Lemma \ref{lemma_l} and since the sign of $(\beta_n)$
alternates by~\eqref{pb}.  
\end{proof}
\bigskip


\section{Convergence and positive probability scenarios}
\label{s_main}

We begin this section by proving Theorem~\ref{th_conv}, which will be just a quick compilation of the previous results. 

\begin{proof}[Proof of Theorem~\ref{th_conv}] We can split the probability space into the following events:
\smallskip

(a) On the event $\mathcal{D}$ we have $\Theta_n\to v\in V$. It follows from~\eqref{iter_lim_th} that $\pi_n+\pi_{n+1}\to \one -v$, 
implying $v_n+v_{n+1}\to 0$. Since $q_{[v]}\in \partial \Sigma$ and $\Vert v_n\Vert \to \ell$ we obtain either scenario (4) of the theorem in $\ell=0$
or scenario (5) if $\ell>0$. 
\smallskip

(b) On the event $\{\chi_{\star}>0\}$ scenario (1) occurs by Proposition~\ref{p_eq}. 
\smallskip

(c) The event $\{\chi_{\star}=0, \chi^{\star}>0\}\cap \mathcal{D}^c$ has probability zero by Propositions~\ref{p_dpos_lpos}
 and~\ref{p_ddd}. 
\smallskip

(d) On the event $\{\chi^{\star}=0, \ell=0\}\cap \mathcal{D}^c$ scenario (2) occurs by Proposition~\ref{p_lpos}. 
\smallskip

(e) On the event $\{\chi^{\star}=0, \ell>0\}\cap \mathcal{D}^c$ scenario (3) occurs by Proposition~\ref{p_lpos}.
\end{proof}
\smallskip

Now we turn our attention to the positive probability scenarios. First, we show that convergence to an internal equilibrium 
is one such case. This is done by artificially forcing the particles 
to choose between edges as equally as possible (largely disregarding the reinforcement probabilities) so that both $\Theta_n$ and $\pi_n$
stay close to the center of the simplex, and then showing that the sequence will not be able to change its behaviour as much as 
to be able to get to the boundary after that.

\begin{prop} 
\label{p_pos1}
$\P(\Theta\in \Sigma\setminus\partial\Sigma, \ell=0)>0$. 
\end{prop}

\begin{proof} 
Let $\sigma=(\frac 1 3, \frac 1 3, \frac 1 3)$. It follows from Lemma~\ref{l_norm} that $\vertiii{M_{\sigma}}=\frac 1 2$. Let $\delta>0$ be such that 
$\vertiii{M_{p}} \le \frac 2 3$ and consequently 
\begin{align}
\label{xx5}
\vertiii{\Big(1-\frac\rho 2\Big)M_p -\frac\rho 2 I}\le  \frac{4+\rho}{6}\le \frac 5 6
\end{align}
for all $p$ satisfying $\Vert p-\sigma\Vert <\delta$. 
\smallskip

Let $1<\nu<\hat\nu<\mu^{\frac 1 2}$ and let $\varpi$ be as in Lemma~\ref{l_rho8}. Let $\gamma\in (0,1)$ be sufficiently close to one
and $\e>0$ be sufficiently small.
Finally, suppose that $\hat m$ is sufficiently large and let $m=4\varpi(\hat m)$. 
\smallskip

Let 
\begin{align*}
\mathcal{E}=\Big\{\Big|B_{n+1}^{\ssup i}-\frac{1}{2}\sum_{k=1}^{N_n^{\ssup i}}Z^{\ssup i}_{n+1,k}\Big|\le \frac 1 2\text{ for all }i \text{ and all }n<m\Big\} 
\end{align*}
be the event that up to the time $m$  
roughly a half of the population of each vertex crosses each of the two adjacent edges, and let  
\begin{align*}
\mathcal{Z}=\Big\{ 
&\Big|\frac{\mu\Vert N_n\Vert}{\Vert N_{n+1}\Vert}-1\Big|\le \nu^{-n},
\Vert \Gamma_{n+1}\Vert\le \nu^{-n},
|\rho^{-1}_{n+1}-\rho^{-1}|\le \nu^{-n},
\Vert N_{n+1}\Vert \ge \hat\nu^{2(n+1)},\text{ for all }\frac m 4 \le n<m\Big\}
\end{align*}
be the event of typical branching between the times $\frac m 4 = \varpi(\hat m)$
and 
$m=4\varpi(\hat m)$. We will see that 
$(\Theta_{m}, \pi_{m})$ is close to $(\sigma,\sigma)$ on the event $\mathcal{E}\cap \mathcal{Z}$. Further, let 
\begin{align*}
\mathcal{C}&=\Big\{
\Big|\frac{\mu\Vert N_n\Vert}{\Vert N_{n+1}\Vert}-1\Big|\le \nu^{-n}, 
\Vert \Gamma_{n+1}\Vert\le \nu^{-n}, 
|\rho_{n+1}^{-1}-\rho^{-1}|\le \nu^{-n}, \Vert \Xi_{n+1}\Vert\le \nu^{-n}\text{ for all }n\ge m\Big\}
\end{align*}
be the event of typical branching and movement after time $m$. We will see that on the event 
\begin{align*}
\mathcal{A}=\mathcal{E}\cap \mathcal{Z}\cap \mathcal{C}
\end{align*}
the process is unable to reach the boundary since it  is too far from it at time $m$
and converges exponentially while away from the boundary. 
\smallskip

{\it Step 1.} Let us show that 
\begin{align*}
\P(\mathcal{A})>0.
\end{align*}

Since $\Vert N_{m}\Vert\ge \hat\nu^{2m}$ 
and $|\rho_m^{-1}-\rho^{-1}|\le \nu^{-(m-1)}$ on $\mathcal{Z}$, 
it follows from Lemmas~\ref{l_n11}, ~\ref{l_xi},~\ref{l_ga} and~\ref{l_rho8}  that
there is $\gamma<1$ such that 
\begin{align*}
 \1_{\mathcal{E}\cap \mathcal{Z}}\P_{\mathcal{F}_{m}}( \mathcal{C}^c)\le \gamma.
\end{align*}
This yields 
\begin{align*}
\P( \mathcal{E}\cap \mathcal{Z}\cap \mathcal{C}^c )
&=\E\Big[ \1_{\mathcal{E}\cap \mathcal{Z}}\P_{\mathcal{F}_{m}}( \mathcal{C}^c)\Big]
\le \gamma \P( \mathcal{E}\cap \mathcal{Z}).
\end{align*}
and therefore  
\begin{align*}
\P(\mathcal{A})
\ge (1-\gamma)\P(\mathcal{E} | \mathcal{Z})\P(\mathcal{Z}).
\end{align*}
We have $\P(\mathcal{Z})>0$ by Lemmas~\ref{l_n8},~\ref{l_n11},~\ref{l_ga}
and~\ref{l_rho8}, and 
$\P(\mathcal{E} |\mathcal{Z})>0$ 
is straightforward. 

\smallskip

{\it Step 2.} Let us show that on $\mathcal{E}\cap\mathcal{Z}$
\begin{align}
\max\big\{\Vert \Theta_{m}-\sigma\Vert, \Vert \pi_{m}-\sigma\Vert\big\}<\e.
\label{max7}
\end{align}

On the event $\mathcal{E}$ for all $n<m$ the iterations~\eqref{def_t} and~\eqref{def_n} can be rewritten as
\begin{align*}
T^{\ssup i}_{n+1}&=T_n^{\ssup i}+
\frac 1 2 \sum_{k=1}^{N_n^{\ssup{i\oplus 1}}}Z^{\ssup{i\oplus 1}}_{n+1,k}
+\frac 1 2\sum_{k=1}^{N_n^{\ssup{i\ominus 1}}}Z^{\ssup{i\ominus 1}}_{n+1,k}
+\e_{n+1}^{\ssup i},\\
N^{\ssup i}_{n+1}&=
\frac 1 2 \sum_{k=1}^{N_n^{\ssup{i\ominus 1}}}Z^{\ssup{i\ominus 1}}_{n+1,k}
+\frac 1 2 \sum_{k=1}^{N_n^{\ssup{i\oplus 1}}}Z^{\ssup{i\oplus 1}}_{n+1,k}
+\delta_{n+1}^{\ssup i},
\end{align*}
where $|\e_{n+1}^{\ssup i}|\le 1$ and $|\delta_{n+1}^{\ssup i}|\le 1$ account for splitting 
a possibly odd number of particles into two almost equal halves. 
Normalising similarly to~\eqref{iter_th2} and~\eqref{iter_pi2} 
we obtain
\begin{align*}
\Theta^{\ssup i}_{n+1}&=(1-\rho_{n+1})\Theta_n^{\ssup i}+
\rho_{n+1}\Big[\frac 1 2\pi_n^{\ssup{i\oplus 1}}
+\frac 1 2\pi_n^{\ssup{i\ominus 1}}
+\frac 1 2\Upsilon_{n+1}^{\ssup{i\oplus 1}}
+\frac 1 2\Upsilon_{n+1}^{\ssup{i\ominus 1}}
\Big]+\frac{\e_{n+1}^{\ssup i}}{\Vert T_{n+1}\Vert},\\
\pi^{\ssup i}_{n+1}&=
\frac 1 2 \pi_n^{\ssup{i\oplus 1}}
+\frac 1 2\pi_n^{\ssup{i\ominus 1}}
+\frac 1 2\Upsilon_{n+1}^{\ssup{i\oplus 1}}
+\frac 1 2\Upsilon_{n+1}^{\ssup{i\ominus 1}}+\frac{\delta_{n+1}^{\ssup i}}{\Vert N_{n+1}\Vert},\end{align*}
which is equivalent to 
\begin{align}
\label{s6}
\Theta_{n+1}=(1-\rho)\Theta_n+\rho M_{\sigma}\pi_n+\hat R_{n+1}
\qquad\text{and}\qquad
\pi_{n+1}=M_{\sigma}\pi_n+\hat S_{n+1},
\end{align}
where 
\begin{align*}
\hat R_{n+1}=(\rho-\rho_{n+1})[\Theta_n-M_{\sigma}\pi_n]
+\rho_{n+1}M_{\sigma}\Upsilon_{n+1}
+\frac{\e_{n+1}}{\Vert T_{n+1}\Vert}
\qquad\text{and}\qquad
\hat S_{n+1}=-\frac 1 2\Upsilon_{n+1}
+\frac{\delta_{n+1}}{\Vert N_{n+1}\Vert}.
\end{align*}
Since $M_{\sigma}\sigma=\sigma$ and $M_{\sigma}v=-\frac 1 2 v$ for any $v\in\mathcal{L}$
we can rewrite~\eqref{s6} with respect to the centre $\sigma$ as  
\begin{align}
\label{h6}
\Theta_{n+1}-\sigma&=(1-\rho)(\Theta_n-\sigma)-\frac{\rho}{2}(\pi_{n}-\sigma)+\hat R_{n+1},\\
\pi_{n+1}-\sigma&=-\frac 1 2(\pi_n-\sigma)+\hat S_{n+1}.
\label{h7}
\end{align}
On the event $\mathcal{Z}$ the error terms are small 
for all $\frac 1 4 m\le n<m$. Indeed, using~\eqref{upga}, $\Vert M_{\sigma}\Vert=1$ and 
\begin{align}
\label{rho_inv}
|\rho_{n+1}-\rho|<|\rho_{n+1}^{-1}-\rho^{-1}| 
\end{align}
we obtain
\begin{align*}
\Vert \hat R_{n+1}\Vert 
&\le 2|\rho_{n+1}^{-1}-\rho^{-1}|
+\Vert \Gamma_{n+1}\Vert 
+\Big|\frac{\mu\Vert N_n\Vert}{\Vert N_{n+1}\Vert}-1\Big|
+ \frac{3}{\Vert N_{n+1}\Vert}
\le 4\nu^{-n}+3\hat\nu^{-2(n+1)}<5\nu^{-n},\\
\Vert \hat S_{n+1}\Vert 
&
\le \frac 1 2 \Vert \Gamma_{n+1}\Vert +\frac 1 2 \Big|\frac{\mu\Vert N_n\Vert}{\Vert N_{n+1}\Vert}-1\Big|+ \frac{3}{\Vert N_{n+1}\Vert}
\le \nu^{-n}+3\hat \nu^{-2(n+1)}\le 2\nu^{-n}.
\end{align*}
On the event $\mathcal{E}\cap \mathcal{Z}$, we combine this with 
~\eqref{h6} and~\eqref{h7} and obtain, for all  $\frac 1 4 m\le n<m$, 
\begin{align*}
\Vert \Theta_{n+1}-\sigma\Vert
&\le (1-\rho)\Vert \Theta_{n}-\sigma\Vert+\frac\rho 2 \Vert \pi_n-\sigma\Vert+5\nu^{-n},\\
\Vert \pi_{n+1}-\sigma\Vert
&\le \frac 1 2 \Vert \pi_{n}-\sigma\Vert + 2\nu^{-n}, 
\end{align*}
which in turn implies
\begin{align*}
\max\big\{\Vert \Theta_{n+1}-\sigma\Vert, \Vert \pi_{n+1}-\sigma\Vert\big\}
\le \Big(1-\frac \rho 2\Big)
\max\big\{\Vert \Theta_{n}-\sigma\Vert, \Vert \pi_{n}-\sigma\Vert\big\}+5 \nu^{-n}
\end{align*}
by using the two-dimensional matrix infinity-norm and the observation that $1-\frac \rho 2\ge \frac 1 2$. Finally, it follows from Lemma~\ref{lemma_notgeom} that
\begin{align*}
\max\big\{\Vert \Theta_{m}-\sigma\Vert, \Vert \pi_{m}-\sigma\Vert\big\}
&\le \gamma^{\frac{3m}{4}}
\Big[\max\big\{\Vert \Theta_{ \frac m 4 }-\sigma\Vert, \Vert \pi_{\frac m 4}-\sigma\Vert\big\}
+ \nu^{- \frac m 4}\frac{5\nu}{\gamma\nu-1}\Big]
\le \gamma^{\frac{3m}{4}}\Big[2+\frac{5\nu}{\gamma\nu-1}\Big]<\e.
\end{align*}

{\it Step 3.}
Let us show that 
\begin{align*}
k=\inf\{n>m: \Vert \Theta_n-\sigma\Vert \ge \delta\}. 
\end{align*}
is infinite on $\mathcal{A}$. 
\smallskip

On the event $\mathcal{E}\cap \mathcal{Z}$ by~\eqref{max7}  we have 
\begin{align}
\label{xx6}
\Vert v_m\Vert \le \Vert \pi_m-\sigma\Vert +\Vert q_{[\Theta_m]}-\sigma\Vert 
<2 \e. 
\end{align}

On the event $\mathcal{C}$ the error terms are small for all $n\ge m$. Indeed, 
using~\eqref{upga} and~\eqref{rho_inv} we have 
\begin{align}
\max\big\{\Vert U_{n+1}\Vert , \Vert W_{n+1}\Vert\big\}
&\le \Vert R_{n+1}\Vert +  \Vert S_{n+1}\Vert 
\le 4 \Vert \Upsilon_{n+1}\Vert +  \Vert \Xi_{n+1}\Vert + 6|\rho_{n+1}-\rho|\notag\\
&\le 4 \Vert \Gamma_{n+1}\Vert +4 \Big|\frac{\mu\Vert N_n\Vert}{\Vert N_{n+1}\Vert}-1\Big|
+  \Vert \Xi_{n+1}\Vert + 6|\rho_{n+1}^{-1}-\rho^{-1}|
\le 15\nu^{-n}.
\label{u9}
\end{align}

On the event $\mathcal{A}$ we obtain for all 
$m\le n<k$ using~\eqref{iter_v},~\eqref{xx5}, ~\eqref{max7} and~\eqref{u9} that 
\begin{align*}
\Vert v_{n+1}\Vert \le \frac 5 6 \Vert v_n\Vert +15\nu^{-n}.
\end{align*}
Similarly to~\eqref{xx4}, we have by Lemma~\ref{lemma_notgeom},~\eqref{xx6} and~\eqref{u9}
that  for any $m\le n<k$
\begin{align}
\Vert \Theta_{n+1}-\Theta_m\Vert 
&\le \frac{2}{1-\gamma}\Big[\Vert v_m\Vert +\nu^{-m}\frac{15\nu}{\gamma \nu-1}\Big] +\frac{15\nu}{\nu-1}\nu^{-m}
<\frac{\delta}{2}.
\end{align}
Together with~\eqref{max7} this gives $\Vert \Theta_{n+1}-\sigma\Vert <\delta$ implying $k=\infty$. 
This yields $\Theta\in \Sigma\setminus\partial\Sigma$ and hence also $\ell=0$ by Theorem~\ref{th_conv}. 
\end{proof}

Now we prove that convergence to an edge equilibrium two-cycle occurs with positive probability. Again, we force the particles 
to adopt a desired behaviour for a while (first staying close to the center of the symplex while the population grows and then not traversing the third edge at all while crossing the other two edges as equally as possible)
and show that the system never recovers from that.

\begin{prop} 
\label{p_pos2}
$\P(\Theta\in \partial\Sigma\setminus V, \ell>0)>0$. 
\end{prop}

\begin{proof}
Let $\sigma=\big(\frac 1 3, \frac 1 3, \frac 1 3\big)$ and 
$\hat\sigma=\big(\frac 1 2, \frac 1 2, 0\big)$. Further, let $\theta=\frac 1 3$.  
\smallskip

Let $1<\nu<\hat\nu<\mu^{\frac 1 2}$ and let $\varpi$ be as in Lemma~\ref{l_rho8}. 
Let $\e>0$ be sufficiently small. Finally, suppose that $\hat m$ is sufficiently large and let $m=4\varpi(\hat m)$. 
\smallskip

Let 
\begin{align*}
\mathcal{E}'&=\Big\{\Big|B_{n+1}^{\ssup i}-\frac{1}{2}\sum_{k=1}^{N_n^{\ssup i}}Z^{\ssup i}_{n+1,k}\Big|\le \frac 1 2\text{ for all }i \text{ and all }n<\frac{m}{2}\Big\} \\
\mathcal{E}''&=\Big\{B_{n+1}^{\ssup 1}=0, \bar B_{n+1}^{\ssup 2}=0, \Big|B_{n+1}^{\ssup 3}-\frac{1}{2}\sum_{k=1}^{N_n^{\ssup 3}}Z^{\ssup 3}_{n+1,k}\Big|\le \frac 1 2\text{ for all }\frac{m}{2}\le n<m\Big\} 
\end{align*}
be the event that up to time $\frac{1}{2}m=2\varpi(\hat m)$  
roughly a half of the population of each vertex crosses each of the two adjacent edges, and 
the event that between the times $\frac{1}{2}m$ and $m$
the third edge is not traversed at all, and 
roughly a half of the population of the third vertex crosses each of the two adjacent edges. Further, 
let $\mathcal{E}=\mathcal{E}'\cap \mathcal{E}''$ and 
let $\mathcal{Z}$ and $\mathcal{C}$ be defined in the same way as in the proof of Proposition~\ref{p_pos1}. We will see that on the event 
\begin{align*}
\mathcal{A}=\mathcal{E}\cap \mathcal{Z}\cap \mathcal{C}
\end{align*}
(a) at time $\frac{1}{2}m$ the population is large and distributed roughly equally between the vertices, which ensures that (b) by the time 
$m$ it gets close to the edge equilibrium two-cycle regime $(\hat\sigma, q_{[\hat\sigma]}\pm \theta e_{-1}(\hat\sigma))$, and  (c) stays at an  equilibrium two-cycle regime forever. 
\smallskip

{\it Step 1.}  We have 
\begin{align*}
\P(\mathcal{A})>0,
\end{align*}
which can be shown in the same way as in {\it Step 1} of the proof of Proposition~\ref{p_pos1}. 
\smallskip

{\it Step 2.} Let $\hat\e>0$ be sufficiently small. 
On the event  $\mathcal{E}'\cap \mathcal{Z}$ 
\begin{align}
\label{centre}
\max\{\Vert\Theta_{\frac{m}{2}}-\sigma\Vert, \Vert\pi_{\frac{m}{2}}-\sigma\Vert\}<\hat\e,
\end{align}
which can be shown analogously to {\it Step 2} in the proof of Proposition~\ref{p_pos1}. 
\smallskip

{\it Step 3.} 
Let us show that on $\mathcal{E}\cap \mathcal{Z}$
\begin{align}
\label{centre1}
\max\big\{\Vert \Theta_m-\hat \sigma\Vert,  
\Vert \pi_m-q_{[\hat\sigma]}-\theta e_{-1}(\hat \sigma)\Vert\big\}<6\hat \e. 
\end{align}

Observe that 
$e_{-1}(\hat\sigma)=\big(\frac 1 4, \frac 1 4, -\frac 1 2\big)$ and 
$q_{[\hat\sigma]}+ \theta e_{-1}(\hat\sigma)=\sigma$. Hence it suffices to show that 
\begin{align}
\label{centre0}
\max\big\{\Vert \Theta_m-\hat \sigma\Vert,  
\Vert \pi_m-\sigma\Vert\big\}<6\hat \e. 
\end{align}

On the event $\mathcal{E}$ for all $\frac 1 2 m \le n<m$ the iterations~\eqref{def_t} and~\eqref{def_n} can be 
rewritten as  
\begin{align*}
\begin{array}{rl}
T_{n+1}^{\ssup 1}\!\!\!\!\! &=T_n^{\ssup 1}+\sum\limits_{k=1}^{N_n^{\ssup 2}}Z_{n+1, k}^{\ssup 2} 
+\frac 1 2 \sum\limits_{k=1}^{N_n^{\ssup 3}}Z_{n+1, k}^{\ssup 3} +\e_{n+1}^{\ssup 1},\\
T_{n+1}^{\ssup 2}\!\!\!\!\! &=T_n^{\ssup 2}
+ \sum\limits_{k=1}^{N_n^{\ssup 1}}Z_{n+1, k}^{\ssup 1} 
+\frac 1 2 \sum\limits_{k=1}^{N_n^{\ssup 3}}Z_{n+1, k}^{\ssup 3} 
+\e_{n+1}^{\ssup 2},\\
T_{n+1}^{\ssup 3}\!\!\!\!\! &=T_n^{\ssup 3},\phantom{+\sum\limits_{k=1}^{N_n^{\ssup 1}}} 
\end{array}
\qquad\text{and}\qquad 
\begin{array}{rl}
N_{n+1}^{\ssup 1}\!\!\!\!\! &=\frac 1 2 \sum\limits_{k=1}^{N_n^{\ssup 3}}Z_{n+1, k}^{\ssup 3}  
+ \delta_{n+1}^{\ssup 1},\\
N_{n+1}^{\ssup 2}\!\!\!\!\! &=\frac 1 2 \sum\limits_{k=1}^{N_n^{\ssup 3}}Z_{n+1, k}^{\ssup 3}  
+ \delta_{n+1}^{\ssup 2},\\
N_{n+1}^{\ssup 3}\!\!\!\!\! &= \sum\limits_{k=1}^{N_n^{\ssup 1}}Z^{\ssup 1} _{n+1, k} 
 +  \sum\limits_{k=1}^{N_n^{\ssup 2}}Z_{n+1, k}^{\ssup 2}, 
\end{array}
\end{align*}
where $|\e_{n+1}^{\ssup i}|\le 1$ and $|\delta_{n+1}^{\ssup i}|\le 1$ account for splitting 
a possibly odd number of particles into two almost equal halves for $i\in\{1,2\}$; we set 
$\e_{n+1}^{\ssup 3}=\delta_{n+1}^{\ssup 3}=0$. 
Normalising similarly to~\eqref{iter_th2} and~\eqref{iter_pi2} 
we obtain
\begin{align}
\label{oo1}
\Theta_{n+1}&=(1-\rho)\Theta_n
+\rho(\one-[M_{\hat\sigma}+I]\pi_n)+\hat R_{n+1}\\
\pi_{n+1}&=M_{\hat\sigma}\pi_n+\hat S_{n+1},
\label{oo5}
\end{align}
 where 
\begin{align*}
\hat R_{n+1}&=(\rho_{n+1}-\rho)(\one-\Theta_n-[M_{\hat\sigma}+I]\pi_n)
-\rho_{n+1}[M_{\hat\sigma}+I]\Upsilon_{n+1}
+\frac{\e_{n+1}}{\Vert T_{n+1}\Vert},\\
\hat S_{n+1}&=M_{\hat\sigma}\Upsilon_{n+1}+\frac{\delta_{n+1}}{\Vert N_{n+1}\Vert}.
\end{align*}
Due to $M_{\hat{\sigma}}q_{[\hat\sigma]}=q_{[\hat\sigma]}$ it follows from~\eqref{oo1} and~\eqref{oo5} that  
\begin{align*}
\left(\begin{array}{c}
\Theta_{n+1}-\hat\sigma\ \ \ \\
\pi_{n+1}-q_{[\hat\sigma]}
\end{array}\right)
=A
\left(\begin{array}{c}
\Theta_{n}-\hat\sigma\ \ \ \\
\pi_{n}-q_{[\hat\sigma]}
\end{array}\right)
+\left(\begin{array}{c}
\hat R_{n+1}\\ \hat S_{n+1}
\end{array}\right), 
\qquad\text{where}\quad 
A=\left(\begin{array}{cc}
(1-\rho)I & -\rho[M_{\hat\sigma}+I]\\
0 & M_{\hat\sigma}
\end{array}\right).
\end{align*}
Iterating, we obtain 
\begin{align}
\label{oo8}
\left(\begin{array}{c}
\Theta_m-\hat\sigma\ \ \ \\
\pi_m-q_{[\hat\sigma]}
\end{array}\right)
&=A^{\frac{m}{2}}
\left(\begin{array}{c}
\Theta_{\frac{m}{2}}-\hat\sigma\ \ \ \\
\pi_{\frac{m}{2}}-q_{[\hat\sigma]}
\end{array}\right)+\sum_{n=m/2}^{m-1}
A^{m-n-1}
\left(\begin{array}{c}
\hat R_{n+1}\\ \hat S_{n+1}
\end{array}\right).
\end{align}
It is easy to see that for any $n$
\begin{align}
\label{oo6}
A^n=\left(\begin{array}{cc}
(1-\rho)^nI & B_n\\
0 & M_{\hat\sigma}^n
\end{array}\right),
\qquad\text{ where} \quad
B_n=-\rho[M_{\hat\sigma}+I]\sum\limits_{i=0}^{n-1}M_{\hat\sigma}^i(1-\rho)^{n-i-1}.
\end{align}
Observe that since
$\Vert M_{\hat\sigma}\Vert=1$, we  have
\begin{align}
\label{oo9}
\Vert B_n\Vert \le 2\rho\sum\limits_{i=0}^{n-1}(1-\rho)^{n-i-1}\le 2.
\end{align}

First, by $M_{\hat\sigma}^2\sigma=\sigma$, 
$M_{\hat\sigma}q_{[\hat\sigma]}=q_{[\hat\sigma]}$, 
$\Vert M_{\hat\sigma}\Vert=1$, \eqref{centre}, \eqref{oo8}, 
~\eqref{oo6}, and due to $\frac 1 2 m=2\varpi(\hat m)$ being even   
\begin{align}
\Vert\pi_m-\sigma\Vert
&=\Big\Vert M_{\hat\sigma}^{\frac m 2}(\pi_{\frac{m}{2}}-\sigma)
+\sum_{n=m/2}^{m-1}M_{\hat\sigma}^{m-n-1}\hat S_{n+1}\Big\Vert
\le \Vert \pi_{\frac{m}{2}}-\sigma\Vert
+\sum_{n=m/2}^{m-1}\Vert \hat S_{n+1}\Vert\notag\\
&\le \hat\e+\sum_{n=m/2}^{m-1}\Big[\Vert \Upsilon_{n+1}\Vert +  \frac{2}{\Vert N_{n+1}\Vert} \Big]
\le  \hat\e+\sum_{n=m/2}^{m-1}\Big[\Vert \Gamma_{n+1}\Vert
+\Big|\frac{\mu\Vert N_n\Vert}{\Vert N_{n+1}\Vert}-1\Big|
+ \frac{2}{\Vert N_{n+1}\Vert} \Big]\notag\\
&\le \hat\e+2 \sum_{n=m/2}^{m-1} \big[\nu^{-n} + \hat\nu^{-2(n+1)}\big]<2\hat\e. 
\label{oo11}
\end{align}

Second, it follows from \eqref{oo8} and~\eqref{oo6}
that
\begin{align}
\label{oo12}
\Theta_m-\hat\sigma
&= (1-\rho)^{\frac{m}{2}}(\Theta_{\frac{m}{2}}-\hat\sigma)
+B_{\frac{m}{2}}(\pi_{\frac{m}{2}}-q_{[\hat\sigma]})
+\sum_{n=m/2}^{m-1}(1-\rho)^{m-n-1}\hat R_{n+1}
+\sum_{n=m/2}^{m-1}B_{m-n-1}\hat S_{n+1}.
\end{align}
For the first term, we observe that 
\begin{align}
\label{oo10}
\Vert(1-\rho)^{\frac{m}{2}}(\Theta_{\frac{m}{2}}-\hat\sigma)\Vert 
\le 2(1-\rho)^{\frac{m}{2}}< \hat \e. 
\end{align}
For the second term, we have $[M_{\hat\sigma}+I](\sigma-q_{[\hat\sigma}])=0$. Combining this with 
~\eqref{oo6} we obtain 
\begin{align*}
B_{\frac{m}{2}}(\pi_{\frac{m}{2}}-q_{[\hat\sigma]})
=B_{\frac{m}{2}}(\pi_{\frac{m}{2}}-\sigma)+B_{\frac{m}{2}}(\sigma-q_{[\hat\sigma]})
=B_{\frac{m}{2}}(\pi_{\frac{m}{2}}-\sigma),
\end{align*}
and it follows now from~\eqref{centre} and~\eqref{oo8} that 
\begin{align}
\label{oo9}
\Vert B_{\frac{m}{2}}(\pi_{\frac{m}{2}}-q_{[\hat\sigma]})\Vert <2 \hat\e.
\end{align}
For the third term, we observe that on $\mathcal{Z}$, using $\Vert M_{\hat\sigma}\Vert=1$, ~\eqref{upga} and~\eqref{rho_inv}, we have for all $\frac{1}{2} m\le n<m$
\begin{align*}
\Vert \hat R_{n+1}\Vert\le  5|\rho_{n+1}^{-1}-\rho^{-1}|+2\Vert \Gamma_{n+1}\Vert
+2\Big|\frac{\mu\Vert N_n\Vert}{\Vert N_{n+1}\Vert}-1\Big| + \frac{2}{\Vert N_{n+1}\Vert}
\le 9\nu^{-n}+2\hat\nu^{-2(n+1)}<\rho\hat\e,
\end{align*}
and hence 
\begin{align}
\label{oo13}
\Big\Vert \sum_{n=m/2}^{m-1}(1-\rho)^{m-n-1}\hat R_{n+1}\Big\Vert 
<\rho\hat\e\sum_{i=0}^{\infty}(1-\rho)^i=\hat \e.  
\end{align}
Finally, the fourth term is treated similarly to~\eqref{oo11}
, obtaining
\begin{align}
\label{oo14}
\Big\Vert \sum_{n=m/2}^{m-1}B_{m-n-1}\hat S_{n+1}\Big\Vert 
\le 2 \sum_{n=m/2}^{m-1}\Vert\hat S_{n+1}\Vert < 2\hat\e.
\end{align}
Substituting~\eqref{oo10}, \eqref{oo9}, \eqref{oo13} and~\eqref{oo14} into~\eqref{oo12}
we obtain 
\begin{align*}
\Vert\Theta_m-\hat\sigma \Vert < 6\hat\e. 
\end{align*}
Combining this with~\eqref{oo11} we obtain~\eqref{centre0} as required.
\smallskip

{\it Step 4.} It will follow now from Lemma~\ref{pall} that $\Theta\in \partial \Sigma$ and $\ell>0$ 
on the event $\mathcal{A}$. It just remains to show that the assumptions
~\eqref{inside0} and~\eqref{inside_error} of the lemma are satisfied. 
\smallskip

Let $\bar\nu<\nu$ (which will be used in the lemma instead of $\nu$), and let us choose the rest of the parameters $\delta, M, \e_0, \gamma,\eta,\e$ as in the lemma. Observe that~\eqref{inside0} follows from~\eqref{centre1} if $\hat\e$ is sufficiently small. Further, on the event $\mathcal{C}$ we have 
obtain similarly to~\eqref{u9} that 
\begin{align*}
\Vert R_{n+1}\Vert +\Vert S_{n+1}\Vert
\le 15 \nu^{-n}<\bar\nu^{-n}<\e^4
\end{align*}
for all $n\ge m$ since $m$ is sufficiently large. 
\end{proof}
\smallskip

\begin{proof}[Proof of Theorem~\ref{th_pos}] The statements follow immediately from Propositions~\ref{p_pos1} and~\ref{p_pos2}.
\end{proof}

\bigskip


\section{No monopoly}
\label{s_mon}

In this section we discuss why the number of traversals of each edge grows super-polynomially. The main technique here is to 
split the number of crossings into a martingale part and an increasing process, and show that the former is controlled by the latter. 
We can then obtain the desired growth from that of the increasing process, which is easier to tackle due to its monotonicity and lower level of randomness.   

\begin{proof}[Proof of Theorem~\ref{th_mon}] 
We will show that for each $i$ and $m\in \N$, $T^{\ssup i}_n>n^m$ eventually almost surely.
\smallskip

Let $i$ be fixed. 
By~\eqref{def_t} we have for all $n$
\begin{align*}
T_{n+1}^{\ssup i}
=T_0^{\ssup i}+\sum_{k=0}^n(B_{k+1}^{\ssup{i\oplus 1}}+\bar B_{k+1}^{\ssup{i\ominus 1}})
=T_0^{\ssup i}+M_{n+1}+Y_{n+1},
\end{align*}
with  
\begin{align*}
M_n=\sum_{k=0}^{n-1}C_{k+1}^{\ssup{i\oplus 1}}+\sum_{k=0}^{n-1}\bar C_{k+1}^{\ssup{i\ominus 1}}
\qquad \text{and}\qquad 
Y_n=\sum_{k=0}^{n-1}\Big[
P_k^{\ssup{i\oplus 1}}\sum_{j=1}^{N_k^{\ssup{i\oplus 1}}}Z^{\ssup{i\oplus 1}}_{k+1,j}
+\bar P_k^{\ssup{i\ominus 1}}\sum_{j=1}^{N_k^{\ssup{i\ominus 1}}}Z^{\ssup{i\ominus 1}}_{k+1,j}\Big],
\end{align*}
where $(C_{n+1})$ are centralised binomials wee defined in~\eqref{cb}
and $(\bar C_{n+1})$ are defined similarly by 
\begin{align*}
\bar{C}_{k+1}^{\ssup i}=\bar{B}_{k+1}^{\ssup i}-\bar{P}_k^{\ssup i}\sum_{j=1}^{N_k^{\ssup i}}Z^{\ssup i}_{k+1,j}.
\end{align*}
Observe that $(M_n)$ is a martingale with respect to $(\mathcal{F}_n^{\star})$ and its quadratic variation satisfies 
\begin{align}
\label{my}
\langle M\rangle_n=\sum_{k=0}^{n-1}\Big[
(1-P_k^{\ssup{i\oplus 1}})P_k^{\ssup{i\oplus 1}}\sum_{j=1}^{N_k^{\ssup{i\oplus 1}}}Z^{\ssup{i\oplus 1}}_{k+1,j}
+(1-\bar P_k^{\ssup{i\ominus 1}})\bar P_k^{\ssup{i\ominus 1}}\sum_{j=1}^{N_k^{\ssup{i\ominus 1}}}Z^{\ssup{i\ominus 1}}_{k+1,j}\Big]\le Y_n, 
\end{align}
and $(Y_n)$ is an increasing previsible process with respect to $(\mathcal{F}_n^{\star})$. 
\smallskip

If the quadratic variation $\langle M\rangle_n$ converges then $M_n$ converges and hence $T_n^{\ssup i}$ grows at the same rate as $Y_n$; 
if $\langle M\rangle_n\to \infty$ then $\frac{M_n}{\langle M\rangle_n}\to 0$, see~\cite[\S 12.14]{Will91}
. 
By~\eqref{my} we have $\frac{M_n}{Y_n}\to 0$, which also implies that $T_n^{\ssup i}$ grows at the same rate as $Y_n$. 
\smallskip

Observe that since $P(Z=0)=0$
we have almost surely
\begin{align}
\label{lb_m}
Y_n\ge  \sum_{k=0}^{n-1}\big[
P_k^{\ssup{i\oplus 1}}N_k^{\ssup{i\oplus 1}}
+\bar P_k^{\ssup{i\ominus 1}}N_k^{\ssup{i\ominus 1}}\big]. 
\end{align}

We have using Lemma~\ref{l_rho8}
\begin{align}
\label{th_m}
\Theta_{k}^{\ssup i}=\frac{T_k^{\ssup i}}{\Vert T_k\Vert }
\ge \frac{\rho_k T_0^{\ssup i}}{\Vert N_k\Vert }
\ge \frac{\rho T_{0}^{\ssup i}}{2\Vert N_k\Vert }
\end{align}
and hence
\begin{align}
\label{p_m}
P_k^{\ssup{i\oplus 1}}\ge \frac{\rho T_{0}^{\ssup i}}{2\Vert N_k\Vert }
\qquad\text{and}\qquad
\bar P_{k}^{\ssup{i\ominus 1}}\ge \frac{\rho T_{0}^{\ssup i}}{2\Vert N_k\Vert }
\end{align}
eventually for all $k$. 
Further, eventually
for all $k$, we have  
\begin{align}
\label{n1_m}
\text{either}\quad
N_k^{\ssup{i\oplus 1}}+N_k^{\ssup{i\ominus 1}}\ge \frac 1 2 \Vert N_k\Vert  
\quad\text{or}\quad N_k^{\ssup i}\ge  \frac 1 2 \Vert N_k\Vert,
\end{align}
the latter implying  by Lemma~\ref{l_n11} that 
\begin{align}
\label{n2_m}
 N_{k+1}^{\ssup{i\oplus 1}}+N_{k+1}^{\ssup{i\ominus 1}}\ge  \frac 1 2 \Vert N_k\Vert \ge \frac{1}{4\mu}\Vert N_{k+1}\Vert.  
 \end{align}
 It follows now from~\eqref{p_m}, \eqref{n1_m} and~\eqref{n2_m} that eventually
\begin{align*}
&N_{k+1}^{\ssup{i\oplus 1}}P_{k+1}^{\ssup{i\oplus 1}}+N_{k+1}^{\ssup{i\ominus 1}}\bar P_{k+1}^{\ssup{i\ominus 1}}
+N_k^{\ssup{i\oplus 1}}P_k^{\ssup{i\oplus 1}}+N_k^{\ssup{i\ominus 1}}\bar P_k^{\ssup{i\ominus 1}}
\ge \frac{\rho T_{0}^{\ssup i}}{8\mu},
\end{align*}
which yields by~\eqref{lb_m} that $(Y_n)$ and hence $T_n^{\ssup i}$ grow at least linearly. 
\smallskip

Let us show by induction over $m$ that $(T_n^{\ssup i})$ grows at least as $n^m$. Suppose this is true for some $m$ and thus there is $c$ such that 
$T_n^{\ssup i}\ge c n^m$ eventually for all $n$. Similarly to~\eqref{th_m} and~\eqref{p_m} we have 
\begin{align}
\label{th_m2}
\Theta_{k}^{\ssup i}=\frac{T_k^{\ssup i}}{\Vert T_k\Vert }
\ge \frac{c\rho_k k^m}{\Vert N_k\Vert }
\ge \frac{c\rho k^m}{2\Vert N_k\Vert }
\end{align}
and hence
\begin{align}
\label{p_m2}
P_k^{\ssup{i\oplus 1}}\ge\frac{c\rho k^m}{2\Vert N_k\Vert }
\qquad\text{and}\qquad
\bar P_{k}^{\ssup{i\ominus 1}}\ge \frac{c\rho k^m}{2\Vert N_k\Vert }
\end{align}
eventually for all $k$. 
It follows now from~\eqref{n1_m}, \eqref{n2_m} and~\eqref{p_m2} that  
\begin{align*}
&N_{k+1}^{\ssup{i\oplus 1}}P_{k+1}^{\ssup{i\oplus 1}}+N_{k+1}^{\ssup{i\ominus 1}}\bar P_{k+1}^{\ssup{i\ominus 1}}
+N_k^{\ssup{i\oplus 1}}P_k^{\ssup{i\oplus 1}}+N_k^{\ssup{i\ominus 1}}\bar P_k^{\ssup{i\ominus 1}}
\ge \frac{c\rho k^m}{8\mu},
\end{align*}
which yields by~\eqref{lb_m} that $(Y_n)$ and hence $T_n^{\ssup i}$ grow at least as $n^{m+1}$. 
\end{proof}
\bigskip

%
%

\section*{Aknowledgements}
This research was partially supported by UCL Studentships, the EPSRC grants EP/W005573/1 and EP/X021696/1. 
\bibliographystyle{abbrv}
\bibliography{gwbib}

\begin{thebibliography}{10}

\bibitem{AngCrawKoz14}
O.~Angel, N.~Crawford, and G.~Kozma.
\newblock Localization for linearly reinforced random walks.
\newblock {\em Duke Math. J.}, 163(5):889--921, 2014.

\bibitem{AthNey70}
K.~B. Athreya and P.~E. Ney.
\newblock {\em Branching Processes}.
\newblock Springer, 1970.

\bibitem{CopDiac87}
D.~Coppersmith and P.~Diaconis.
\newblock Random walks with reinforcement.
\newblock Unpublished manuscript, 1987.

\bibitem{CotThac17}
C.~Cotar and D.~Thacker.
\newblock Edge- and vertex-reinforced random walks with super-linear
  reinforcement on infinite graphs.
\newblock {\em Ann. Probab.}, 45(4):2655--2706, 2017.

\bibitem{Dav90}
B.~Davis.
\newblock Reinforced random walk.
\newblock {\em Probab. Th. Rel. Fields}, 84(2):203--229, 1990.

\bibitem{Diac88'}
P.~Diaconis.
\newblock Recent progress on de {F}inetti's notion of exchangeability.
\newblock In {\em Bayesian Statistics 3: Proc. 3rd Valencia Int. Meeting},
  pages 111--125. Oxford University Press, 1987.

\bibitem{DriFriMitz02}
E.~Drinea, A.~Frieze, and M.~Mitzenmacher.
\newblock Balls in bins process with feedback.
\newblock In {\em Proc. 13th ACM-SIAM Symposium on Discrete Algorithms (SODA)},
  pages 308--315. SIAM, 2002.

\bibitem{Giam23}
G.~Giambartolomei.
\newblock {\em The edge-reinforced branching random walk on the triangle and
  generalised balls and bins with positive feedback}.
\newblock Phd thesis, UCL, 2023.

\bibitem{Koz14}
G.~Kozma.
\newblock Reinforced random walk.
\newblock In {\em 6th European Congress of Mathematics (ECM)}, pages 429--443.
  EMS, 2014.

\bibitem{Lim03}
V.~Limic.
\newblock Attracting edge property for a class of reinforced random walks.
\newblock {\em Ann. Probab.}, 31(3):1615--1654, 2003.

\bibitem{MerRoll06}
F.~Merkl and S.~Rolles.
\newblock Linearly edge-reinforced random walks.
\newblock In {\em Dynamics \& stochastics: Festschrift in honor of {M. S.
  K}eane}, volume~48 of {\em IMS Lecture Notes Monogr. Ser.}, pages 66--77.
  IMS, 2006.

\bibitem{MerRoll07}
F.~Merkl and S.~Rolles.
\newblock A random environment for linearly edge reinforced random walks on
  inifinite graphs.
\newblock {\em Probab. Th. Rel. Fields}, 138(1):157--176, 2007.

\bibitem{Pem88'}
R.~Pemantle.
\newblock Phase transition of reinforced random walk and {RWRE} on trees.
\newblock {\em Ann. Prob}, 16(3):1229--1241, 1988.

\bibitem{Pem07}
R.~Pemantle.
\newblock A survey of random processes with reinforcement.
\newblock {\em Probab. Surveys}, 4:1--79, 2007.

\bibitem{Roll03}
S.~Rolles.
\newblock How edge-reinforced random walk arises naturally.
\newblock {\em Probab. Th. Rel. Fields}, 126(2):243--260, 2003.

\bibitem{Roll06}
S.~Rolles.
\newblock On the recurrence of edge-reinforced random walk on $\mathbb{Z}\times
  g$.
\newblock {\em Probab. Th. Rel. Fields}, 135(2):216--264, 2006.

\bibitem{SabTar15}
C.~Sabot and P.~Tarr\`es.
\newblock Edge-reinforced random walk, {V}ertex-{R}einforced {J}ump {P}rocess
  and the supersymmetric hyperbolic sigma model.
\newblock {\em J. Eur. Math. Soc.}, 17(9):2353--2378, 2015.

\bibitem{Sel94}
T.~Sellke.
\newblock Reinforced random walk on the $d$-dimensional integer lattice.
\newblock Technical Report 94-26, Purdue University, 1994.

\bibitem{Sid18}
N.~Sidorova.
\newblock Time-dependent balls and bins model with positive feedback.
\newblock \href{https://arxiv.org/abs/1809.02221}{arXiv:1809.02221 [math.PR]}.

\bibitem{Will91}
D.~Williams.
\newblock {\em Probability with martingales}.
\newblock Cambridge University Press, 1991.

\end{thebibliography}


\begin{thebibliography}{99}
\bibitem{angel} O.\ Angel, N. Crawford and G. Kozma. Localization for linearly reinforced random walks. \emph{Duke Math. J.} 163(5), 889--921, 2014
\bibitem{bp} K.B.\ Athreya and P.\ E.\ Ney. Branching Processes. Springer, 1970
\bibitem{diaconis} D.\ Coppersmith and P.\ Diaconis. Random walks with reinforcement. \emph{Unpublished manuscript}, 1986
\bibitem{davis} B.\ Davis. Reinforced random walk. \emph{Probab.\ Th.\ Rel.\ Fields} 84, 203--229, 1990
\bibitem{kr} M.\ Keane and S.\ Rolles. Edge-reinforced random walk on finite graphs. In: P.\ Cl\' ement, F.\ Hollander, den, J.\ Neerven, van, B.\ Pagter, de (Eds.), Infinite dimensional stochastic analysis, 217--234, 2000
\bibitem{pem_tree} R.\ Pemantle. Phase transition of reinforced random walk and RWRE on trees. \emph{Ann.\ Prob.} 16, 1229--1241, 1988.
\bibitem{pemantle} R.\ Pemantle. A survey of random processes with reinforcement. \emph{Probab.\ Surv.} 4, 1--79, 2007 
\bibitem{sabot} C.\ Sabot and P.\ Tarr\` es. Edge-reinforced random walk, vertex-reinforced jump process and the supersymmetric hyperbolic sigma model. \emph{J.\ Eur.\ Math.\ Soc.} 17, 2353--2378, 2015
\bibitem{williams} D.\ Williams, Probability with martingales, Cambridge University Press, 1991
\end{thebibliography}

\end{document}